\documentclass[12pt]{article}

\usepackage[T1]{fontenc}
\usepackage[utf8]{inputenc}
\usepackage{bbm}
\usepackage{dsfont}
\usepackage{mathrsfs}

\usepackage{enumerate}
\usepackage{amsmath,amssymb,amsthm}
\usepackage{stackrel}
\usepackage{mathtools}
\usepackage{graphicx}
\usepackage[normalem]{ulem} 

\usepackage{authblk}
\usepackage{cite}

\usepackage[hidelinks,bookmarksnumbered]{hyperref}
\hypersetup{pdfstartview=FitH}
\usepackage{tikz}
\usetikzlibrary{patterns}
\usetikzlibrary{calc}
\usetikzlibrary{intersections}
\usetikzlibrary{decorations.pathmorphing}
\usetikzlibrary{through}
\usetikzlibrary{backgrounds}

  \newcounter{constant}
  \newcommand{\newconstant}[1]{\refstepcounter{constant}\label{#1}}
  \newcommand{\uc}[1]{c_{\textnormal{\tiny \ref{#1}}}}
\def\arraypar#1{\parbox[c]{.8\textwidth}{\centering #1}}

\usepackage{environ}
\NewEnviron{display}{
  \begin{equation}\begin{array}{c}
    \arraypar{\BODY}
  \end{array}\end{equation}
}

\newcommand{\I}{\mathds{1}}

\newcommand{\distr}{\ensuremath{\stackrel{\scriptstyle d}{=}}}

\newcommand{\Vol}{\mathop{\mathrm{Vol}}\nolimits}
\newcommand{\diam}{\mathop{\mathrm{diam}}\nolimits}

\newcommand{\badC}{\textup{\textbf{C}}^{\ast}}
\newcommand{\bbadC}{\smash{\bar{\textup{\textbf{C}}}}^{\ast}}
\newcommand{\tbadC}{\smash{\tilde{\textup{\textbf{C}}}}^{\ast}}
\newcommand{\hLambda}{\smash{\widehat{\Lambda}}}
\newcommand{\ZB}{{\mathsf B}}
\newcommand{\comp}{{\mathsf c}}

\newcommand{\cB}{\ensuremath{\mathcal{B}}}

\newcommand{\cD}{\ensuremath{\mathscr{D}}}
\newcommand{\cE}{\ensuremath{\mathcal{E}}}

\newcommand{\cH}{\ensuremath{\mathcal{H}}}

\newcommand{\cP}{\ensuremath{\mathcal{P}}}

\newcommand{\cW}{\ensuremath{\mathcal{W}}}

\newcommand{\EE}{\ensuremath{\mathbb{E}}}

\newcommand{\NN}{\ensuremath{\mathbb{N}}}
\newcommand{\PP}{\ensuremath{\mathbb{P}}}

\newcommand{\RR}{\ensuremath{\mathbb{R}}}
\newcommand{\ZZ}{\ensuremath{\mathbb{Z}}}

\usepackage[marginpar=65pt]{geometry}
\usepackage{setspace}
\theoremstyle{plain}
\newtheorem{prop}{Proposition}
\newtheorem{teo}{Theorem}
\newtheorem{lema}{Lemma}
\newtheorem{coro}{Corollary}
\newtheorem*{teo:vacant_crossing_of_boxes_alpha_2}{Theorem \ref{teo:vacant_crossing_of_boxes_alpha_2}}
\newtheorem*{teo:alpha_phase_transition}{Theorem \ref{teo:alpha_phase_transition}}

\theoremstyle{definition}
\newtheorem{defi}{Definition}

\theoremstyle{remark}
\newtheorem{remark}{Remark}

\DeclareUnicodeCharacter{1EF3}{\`{y}}

\title{Euclidean and chemical distances in ellipses percolation}
\author[1]{Marcelo Hil\'ario}
\author[2]{Daniel Ungaretti}
\affil[1,2]{Universidade Federal de Minas Gerais}
\affil[2]{Universidade de S\~ao Paulo}

\begin{document}
\maketitle

\begin{abstract}
The ellipses model is a continuum percolation process in which ellipses with
random orientation and eccentricity are placed in the plane according to a
Poisson point process. A parameter $\alpha$ controls the tail distribution
of the major axis' distribution and we focus on the regime $\alpha \in
(1,2)$ for which there exists a unique infinite cluster of ellipses and
this cluster fulfills the so called highway property. We prove that the
distance within this infinite cluster behaves asymptotically like the
(unrestricted) Euclidean distance in the plane.  We also show that the
chemical distance between points $x$ and $y$ behaves roughly as
$c \log\log |x-y|$.
\end{abstract}



    



\section{Introduction}
\label{sec:introduction}

In this paper we study both the chemical and Euclidean distances in the
\textit{ellipses model} introduced in~\cite{teixeira2017ellipses}. It is a
Boolean percolation in the plane with defects given by random ellipses centered
at points given by a Poisson point process with intensity $u > 0$. Given the
position of the centers, the eccentricities and orientations of the ellipses
are independent. The minor axes always have length one and they make uniformly
distributed angles with the horizontal direction. The lengths of the major axes
are drawn independently from a heavy-tailed distribution $\rho$ supported on
$[1, \infty)$ that satisfies $\rho[r,\infty) = c r^{-\alpha}$ for $r \geq 1$.
Therefore, while the parameter $u$ controls the amount of ellipses appearing in
the picture, the parameter $\alpha$ controls how eccentric they are.

In~\cite{teixeira2017ellipses}, phase transition and connectivity properties
for the ellipses model were studied as functions of these two parameters.  Here
we will focus on $\alpha \in (1,2)$, the regime in which, for any choice of $u
> 0$, there exists a unique infinite cluster of ellipses that, in addition,
satisfies what we refer to as the \textit{highway property}. Roughly, it means
that after scaling the probability of connecting two regions using a single
ellipse becomes close to one.

Let $\cD(x, y)$ denote the minimum length of a polygonal path from $x$ to $y$
which lies entirely inside the set covered by the ellipses.  We call $\cD
(\cdot \, , \cdot)$ the \emph{Euclidean distance} restricted to the set of
ellipses or sometimes the \emph{internal distance}.  Also, for any two points
$x$ and $y$ in the infinite cluster of ellipses, denote by $D(x,y)$ the
\textit{chemical distance} between them, i.e.\ the minimum number of ellipses
that a continuous path from $x$ to $y$ contained entirely inside the cluster of
ellipses has to intersect.  The Euclidean distance in the plane, sometimes
called the unrestricted Euclidean distance, is denoted by $|x-y|$.

Let us now state our main results 
\begin{teo}[Euclidean distance]
\label{teo:distance_ellipses_l2}
Consider the ellipses model with parameters $u>0$ and $\alpha \in (1, 2)$.
For $x, y \in \RR^{2}$ and $\delta > 0$, 
\begin{equation}
\label{eq:distance_ellipses}
\lim_{|x-y| \to \infty}
    \PP\Bigl(
    1 \le \frac{\cD(x,y)}{|x-y|} \le 1 + |x-y|^{\frac{\alpha-2}{4} + \delta}\,
    \Bigm| \, x \leftrightarrow y
    \Bigr)
    = 1.
\end{equation}
\end{teo}

\begin{teo}[Chemical distance]
\label{teo:chem_distance_ellipses}
Consider the ellipses model with parameters $u>0$ and $\alpha \in (1, 2)$.
For $x, y \in \RR^{2}$ and $\delta > 0$, 
\begin{equation}
\label{eq:chem_distance_ellipses}
\lim_{|x-y| \to \infty} \PP\Bigl(
    \frac{1-\delta}{\log (\frac{2}{\alpha - 1})}
        \le \frac{D(x,y)}{\log \log |x-y|}
        \le \frac{2 + \delta}{\log (\frac{2}{\alpha})} \, \Bigm| \, x \leftrightarrow y
    \Bigr)
    = 1.
\end{equation}
\end{teo}

To understand geometric properties of infinite clusters is
a problem of major interest in percolation theory. Models for which the
chemical distance was studied include Bernoulli percolation and first-passage
percolation~\cite{antal1996chemical, garet2004asymptotic,
garet2007large}; random interlacements~\cite{cerny2012internal}; random walk
loop soup~\cite{chang2017supercritical}; and Gaussian free
field~\cite{ding2018chemical, drewitz2014chemical}. General conditions for a
percolation model on $\ZZ^{d}$ to have a unique infinite cluster in which
Euclidean and chemical distances are comparable are provided
in~\cite{drewitz2014chemical}.

Theorems~\ref{teo:distance_ellipses_l2} and~\ref{teo:chem_distance_ellipses}
show that ellipses model does not fit into the conditions
of~\cite{drewitz2014chemical}. This is due to the presence of long ellipses.
A similar behavior can be observed in Poisson cylinder
model~\cite{tykesson2012percolation} and long-range
percolation~\cite{newman1986one, biskup2004scaling}, as we discuss next.

\subsection{Comparing with long-range models}
\label{sub:comparing_long_range_models}

Ellipses model is closely related to other two percolation models that allow for
arbitrarily long connections: \textit{Poisson cylinders model} on
$\RR^d$ and \textit{long-range percolation} on $\ZZ^{d}$. In principle
one could try to leverage these relations in order to obtain estimates for the
distances in ellipses model, and indeed some of our results are obtained this
way.

We emphasize that the highway property is shared by these three models, with
the immediate adaptations that connection of far away regions is accomplished 
using a single cylinder or a single open edge for Poisson cylinders and long-range model, respectively. 
The highway property is the main tool to ensure that, in all three models, the distance
inside the infinite cluster is asymptotically equivalent to the unrestricted
Euclidean distance in the plane.

However, the behavior of the chemical distance differs completely in each of
these models. 
Before elaborating on these differences, we give a quick
introduction to Poisson cylinder model and long-range percolation.

\medskip
\noindent
\textbf{Poisson Cylinders.}
Poisson cylinders model consists of a random collection of bi-infinite
cylinders of radius one whose axes are given by a Poisson point process on the
space of all the lines (i.e.\ affine one-dimensional subspaces) in $\RR^{d}$
with $d \geq 3$, see~\cite{tykesson2012percolation} for details.
Distances within clusters of cylinders were studied
in~\cite{broman2016connectedness} and ~\cite{hilario2019shape}.
In~\cite{broman2016connectedness} the authors prove that almost surely any two
cylinders are linked by a sequence composed of at most $d-2$ other intersecting
cylinders, implying that the chemical distance is bounded.
For the Euclidean distance on the other hand, in~\cite{hilario2019shape} the
authors prove a shape theorem showing that if $x, y \in \RR^{d}$ are points in
the infinite cluster then the internal distance between $x$ and $y$ is
asymptotically $|x-y| + O(|x-y|^{1/2 + \varepsilon})$, for any $\varepsilon > 0$.

One straightforward connection between Poisson cylinders and ellipses model is
to study the intersection of the random cylinders with any given
$2$-di\-men\-sion\-al plane. As shown in~\cite{ungaretti2017phdthesis} this
intersection is a collection of ellipses whose law is an instance of the
ellipses model with $\alpha=2$ when $d=3$ and, with $\alpha > 2$ in
higher dimensions. Thus, this natural coupling between ellipses model and
Poisson cylinder model is not helpful to draw conclusions when $\alpha$ ranges
in $(1,2)$.


\medskip
\noindent
\textbf{Long-range percolation.}
Fix $\beta, s > 0 $ and consider the bond percolation model, known as
long-range percolation, in which for each $x \neq y \in \ZZ^{d}$ an open edge
connects $x$ and $y$ with probability
\begin{equation*}
p_{xy} = 1 - \exp[ -\beta |x-y|^{-s} ]. 
\end{equation*}
Different expressions for $p_{xy}$ may be considered but it is usually assumed
that it decays roughly as $\beta |x-y|^{-s + o(1)}$ for some positive $\beta$
and $s$.

Let us now explain how long-range percolation and ellipses model relate to each
other. Notice that both models have one parameter that controls the density
($u$ and $\beta$, respectively) and another that controls the distribution of
long connections ($\alpha$ and $s$, respectively). Essentially, a
discretization of ellipses model leads to a long-range percolation with
parameters satisfying the following relations
\begin{equation}
\label{eq:parameters_relation}
s= 2+ \alpha
\quad \text{and} \quad
u = \frac{\pi \beta}{\alpha 2^{\alpha}}.
\end{equation}
The coupling is given as follows.
Take $B:=[-1/2, 1/2)^{2}$ and for $x \in \RR^{2}$ write $B_x = x + B$ so that
$(B_x)_{x \in \mathbb{Z}^2}$ forms a tiling of $\mathbb{R}^2$.  For a
realization of the ellipses model with parameters $u$ and $\alpha$  associate
with every ellipse the two extremities of its major axis.  Now embed
$\mathbb{Z}^2$ in $\mathbb{R}^2$ in the natural way and define two sites $x
\neq y \in \ZZ^{2}$ to be $\xi$-connected and write $x \sim_\xi y$ if there is
an ellipse whose major axis has one extremity in $B_x$ and the other in $B_y$.  
Inserting open
edges between pairs of $\xi$-connected sites leads to a long-range percolation
model whose parameters $s$ and $\beta$ satisfy~\eqref{eq:parameters_relation},
as we show in Section~\ref{sec:preliminaries}.

We will explore this coupling to translate results about long-range percolation
to results about ellipses model. However, there are some key points that must
be dealt with when comparing connectivity in these models using the coupling
described above.

A first issue is that, in some situations, connectivity is favored in ellipses
model.  In fact, when two long ellipses cross each other it may occur that the
resulting open edges in the long-range model belong to different components.
This suggests that connectivity properties in these two models may differ.
Indeed in~\cite[Theorem~1.2]{teixeira2017ellipses} it is shown that in ellipses
model with $\alpha \in (1,2)$ the covered set percolates for any intensity $u>0$.
The corresponding long-range percolation (with $s \in (3,4)$
by~\eqref{eq:parameters_relation}) does not percolate for sufficiently small
$\beta$, since when $\sum_{z\in \ZZ^{2}} p_{0z} < 1$ the open cluster of the
origin is dominated by a subcritical Galton-Watson tree.

A second issue affects connectivity in the opposite direction.  Notice that
having $x \sim_{\xi} y \sim_{\xi} z$ does not ensure that, in the underlying
ellipses model, the corresponding ellipses overlap since it may occur that two
ellipses intersect the box $B_y$ without touching each other, see
Figure~\ref{fig:coupling_fails}.
\begin{figure}[ht]
\centering
\begin{tikzpicture}[scale=1]
\draw (0, 0) rectangle ++(1, 1);
\draw (2, 1) rectangle ++(1, 1);
\draw (-3, 1) rectangle ++(1, 1);
\filldraw[gray, opacity=.3,rotate around={25:(1.5,1.4)}]
    (1.5,1.4) ellipse (1 and .3);
\filldraw[gray, opacity=.3,rotate around={-15:(-1.3, .7)}]
    (-1.3, .7) ellipse (1.6 and .3);
\begin{scope}[shift={(9,0)}]
\draw (0, 0) rectangle ++(1, 1);
\draw (-5, 1) rectangle ++(1, 1);
\filldraw[gray, opacity=.3,rotate around={-10:(-2.3, .7)}]
    (-2.3, .7) ellipse (2.7 and .3);
\draw (-4, 0) rectangle ++(1, 1);
\draw (-2, 1) rectangle ++(1, 1);
\filldraw[gray, opacity=.3,rotate around={30:(-2.7,1)}]
    (-2.7,1) ellipse (1.2 and .3);
\end{scope}
\end{tikzpicture}
\caption{Possible problems when coupling ellipses model and long-range.
    On the left, ellipses that do not intersect lead to a single connected component of $\xi$-edges; on
    the right, ellipses that do intersect lead to disjoint components of $\xi$-edges.}
\label{fig:coupling_fails}
\end{figure}
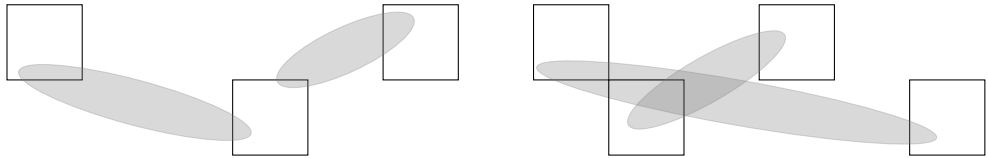

We now present the results about long-range percolation that we use.
We refer the reader
to~\cite[Section~1.3]{biskup2004scaling} for a summary on
the chemical distance for different regimes of $s$.
We will be mainly interested in the case $d=2$ and $s \in (3, 4)$, which
corresponds to ellipses model with $\alpha\in (1,2)$. Results for any $d \geq 2$
and $s \in (d,2d)$ are discussed in
papers~\cite{biskup2004scaling, biskup2011graph, biskup2019sharp}.

Our estimate for Euclidean distance in Theorem~\ref{teo:distance_ellipses_l2}
builds on a construction from~\cite{biskup2004scaling} that relies on the above
mentioned highway property. Let us exemplify this property for the long-range
model with $d \geq 2$ and $s \in (d, 2d)$.  Denote ${|x-y| = N}$ and for
$\gamma \in \bigl( s/2d, 1 \bigr)$ consider $\ZB = \ZZ^{d} \cap [-N^{\gamma}/2,
N^{\gamma}/2]^d$. Then, the probability of the event $\{\ZB_x
\leftrightarrow_1 \ZB_y\}$ that there is an open edge connecting a site in
$\ZB_x = x + \ZB$ to another site in $\ZB_y = y + \ZB$ can be estimated using 
\begin{equation}
\label{eq:highway}
\PP\bigl(\{\ZB_x \leftrightarrow_1 \ZB_y\}^{\comp}\bigr)
    = \prod_{\mathclap{x' \in \ZB_x,y' \in \ZB_y}} (1 - p_{x'y'})
    = \exp\biggl[
            - \frac{\beta |\ZB_x| |\ZB_y|}{\bigl(N + O(N^{\gamma})\bigr)^{s}}
        \biggr] 
    \sim e^{- \beta N^{2d\gamma - s}}
\end{equation}
as $N \to \infty$. 
The estimate in~\eqref{eq:highway} is in the core of the hierarchical
construction from~\cite{biskup2004scaling} which leads to the main result
therein: the chemical distance between two points $x,y$ on the infinite cluster
behaves asymptotically as
\begin{equation*}
D(x,y) = (\log |x-y|)^{\Delta + o(1)}
\quad \text{as $|x-y| \to \infty$,}
\quad \text{where $\Delta=\frac{\log 2}{\log (2d/s)}$.}
\end{equation*}
This same argument shows that ${\cD(x,y) \sim |x-y|}$, although not mentioned
in~\cite{biskup2004scaling}. We present (a simplified version of) their
hierarchical construction in Section~\ref{sec:bounding_euclidean_distance}, and
use it as a fundamental tool for controlling the Euclidean distance traversed
by a path in ellipses model.

\medskip

As we have seen above, in all three models the Euclidean distance restricted to
the covered set and the unrestricted distance are asymptotically the same. 
The chemical distance can be seen as an alternative measure of connectedness for
these models and through this lens they behave very differently, presenting
different orders of magnitude. 
For Poisson cylinders the chemical distance is
bounded by a constant, for ellipses model it grows as $\log \log |x-y|$, and
for the long-range model it grows as $(\log |x-y|)^{\Delta}$. 
We will see that
this discrepancy between ellipses model and long-range percolation may be
explained as a consequence of the first issue above.

\subsection{Idea of proofs}
\label{sub:idea_of_proofs}

We need to obtain lower and upper bounds on the distance between points $x$ and
$y$ that belong to the same cluster of ellipses.

Our estimates for the lower bounds are simpler to obtain. The lower bound for
the internal distance appearing in~\eqref{eq:distance_ellipses} is simply
the unrestricted distance $|x-y|$ in the plane, and although obvious, we do not
have any improvement for it. The lower bound for the chemical distance
appearing in \eqref{eq:chem_distance_ellipses}, follows from an elementary
induction argument. One could try to improve the bounds using the BK inequality
like in the lower bound in~\cite{biskup2011graph} but our argument seems to
provide the correct order of magnitude in a simpler way.

The proofs for the upper bounds appearing in both
Theorems~\ref{teo:distance_ellipses_l2} and~\ref{teo:chem_distance_ellipses}
follow similar strategies. The first step is to show that, with high
probability, there exists a set of few overlapping ellipses that allows us to
traverse from a local region containing $x$ to another local region containing
$y$ without deviating too much. The second step consists of connecting locally
the points $x$ and $y$ to this structure. This is the content of a
deterministic construction in Lemma~\ref{lema:distance_small_scales}. 

Let us discuss some details of each proof. The proof of
Theorem~\ref{teo:distance_ellipses_l2} is based on a coupling of ellipses model
and site-bond long-range percolation model on a renormalized lattice.  The
probability of an edge being open will be given by the coupling with long-range
percolation described above.  Only a subset of the underlying Poisson point
process defining the ellipses model is used for this coupling.  The remaining
(independent) part is used for defining a site percolation model on a lattice
of renormalized sites that correspond to boxes in the original lattice.
Roughly, a site is considered open (or \textit{good}) if the corresponding box
is good meaning that the cluster of ellipses near this box is sufficiently
well-connected. This definition is based on an idea
from~\cite{antal1996chemical}.

On the event that $x \leftrightarrow y$ in ellipses model, the bond percolation
part and the site percolation part are then combined to create a short path
connecting $x$ to $y$. This is done in two steps:
\begin{description}
\item[Hierarchical construction.] This is essentially the construction
    from~\cite{biskup2004scaling} based on the highway
    property~\eqref{eq:highway}.  
    When $|x-y| = N$ is large, with very
    high probability, there is an open edge connecting small
    neighborhoods around $x$ and $y$.
    This idea can be iterated to build what
    we call a \textit{hierarchy}, see Definition~\ref{defi:hierarchy} and
    Figure~\ref{fig:hierarchy}.  
    In words, a hierarchy is a collection of
    long edges (or \textit{highways}) that essentially connects $x$ to $y$,
    leaving only some gaps that are much shorter than the highways.
    
\item[Gluing procedure.]
    Given that we have found a hierarchy the original problem is then replaced
    by the problem of building connections across the remaining gaps.  For
    that, we use the site percolation part of our coupling.  The definition
    of good boxes will ensure that neighboring good boxes have intersecting
    clusters of ellipses.  Moreover, the renormalization scheme is
    performed so that the probability of a box being good is highly
    supercritical. Therefore, even when a gap that we want to cross has
    some bad boxes around it, we can still contour these bad boxes by
    paying a low price in terms of distance and probability.  This is
    accomplished through a large deviation bound on the size of bad
    clusters, see Section~\ref{sub:gluing_highways}.
\end{description}
After these two steps are completed, we have with high probability a path of
ellipses connecting $x$ and $y$ whose length is well-controlled. This
establishes the upper bound in Theorem~\ref{teo:distance_ellipses_l2}.

The reader who is familiar with the hierarchical construction
of~\cite{biskup2004scaling} and the renormalization procedure
of~\cite{antal1996chemical} may notice that, in our proof of
Theorem~\ref{teo:distance_ellipses_l2}, we define events that are much simpler
than the ones appearing in the original constructions. This is possible
due to the existence of long overlapping ellipses that overlap, a phenomenon
with no counterpart in long-range or Bernoulli percolation.

The proof for the upper bound for the chemical distance in
Theorem~\ref{teo:chem_distance_ellipses} does not rely on the same coupling
with long-range percolation as in Theorem~\ref{teo:distance_ellipses_l2} since
this coupling does not explore the possibility of using long ellipses to its
full potential. 
Instead, our argument involves choosing a rapidly increasing sequence
of rectangles and studying the event that they are crossed in the hardest
direction by a single ellipse.  By a Borel-Cantelli argument, this construction
provides `enhanced highways' that cross large distances more efficiently.

\medskip
\noindent
\textbf{Remarks on the notation.} Throughout the paper we use $c, C$ to denote
generic positive constants that can change from line to line. Numbered
constants $\uc{c:cross_one_ellipse}, \uc{c:annulus_bound},
\uc{c:ellipse_length}$, are kept fixed.
Also, our asymptotic notation uses
\begin{itemize}
\setlength{\itemsep}{1.5pt}
\setlength{\parskip}{0pt}
\setlength{\parsep}{0pt}
\item both $f = o(g)$ and $f \ll g$ to denote
    $\lim_{n \to \infty} \tfrac{f(n)}{g(n)} = 0$;
\item $f = O(g)$ to denote $|f| \le C |g|$ for some constant $C$;
\item $f = \Theta(g)$ to denote $c |g| \le |f| \le C |g|$;
\item $f \sim g$ to denote $\lim_{n \to \infty} \tfrac{f(n)}{g(n)} = 1$.
\end{itemize}

\section{Couplings, highways and hierarchies}
\label{sec:preliminaries}

In this Section we collect some results from the literature that will be used in
the proofs of Theorems~\ref{teo:distance_ellipses_l2}
and~\ref{teo:chem_distance_ellipses}.

\medskip
\noindent
\textbf{Ellipses model.} Ellipses model is defined via a Poisson point process
(PPP) on $\RR^{2} \times \RR^{+} \times (-\pi/2, \pi/2]$ with intensity
measure
\begin{equation}
\label{eq:ellipses_measure}
u \cdot \mathrm{d}z
    \otimes \alpha R^{-(1+\alpha)}\, \mathrm{d}R
    \otimes \frac{1}{\pi}\mathrm{d}V,
\end{equation}
For each point $(z, R, V)$ in the PPP, place an ellipse centered at $z$ whose
minor axis has length $1$ and whose major axis has length $R$ and forms an
angle $V$ with respect to the horizontal direction.  The multiplicative
parameter $u >0$ controls the density of ellipse whereas the exponent  $\alpha
>0$ controls the tail of major axis' distribution.  We refer the reader
to~\cite{teixeira2017ellipses} for an account on the phase transition for
percolation on the covered set with respect to parameters $u$ and $\alpha$.

Define the event $LR_1(l;k)$ that an ellipse crosses the box $[0,l] \times
[0,kl]$ from left to right.  The next lemma uncovers the range of parameters in
which ellipses model presents the highway property:
\begin{lema}[Proposition~5.1 of \cite{teixeira2017ellipses}]
\label{lema:cross_box_one_ellipse}
\newconstant{c:cross_one_ellipse}
Let $\alpha >1$. There is a constant
$\uc{c:cross_one_ellipse} = \uc{c:cross_one_ellipse}(\alpha) > 0$
such that for every $k, l > 0 $ with $lk > 2$
\begin{equation}
\label{eq:cross_box_one_ellipse}
1 - e^{ -\uc{c:cross_one_ellipse}^{-1} u (k \wedge k^{-\alpha}) l^{2 - \alpha}}
    \le \PP(LR_1(l;k))
    \le 1 -
        e^{ -\uc{c:cross_one_ellipse} u
        (k^{2-\alpha}\vee k^{-\alpha}) l^{2 - \alpha}}.
\end{equation}
\end{lema}

Therefore, when $\alpha \in (1,2)$ and $k$ is fixed, we have $\PP(LR_1(l;k))
\to 1$ as $l \to \infty$, showing that the highway property holds in this range
of $\alpha$.

A second useful estimate is a similar bound for the probability that there is
an ellipse that traverses an annulus. Let $B(l)$ denote the Euclidean ball of
radius $l$ centered at the origin in $\RR^{2}$ and denote its boundary by
$\partial B(l)$. For two disjoint regions $A_1$ and $A_2$ write $A_1
\leftrightarrow_1 A_2$ if there is one ellipse that intersects both $A_1$ and
$A_2$.  We have:
\begin{lema}
\label{lema:annulus_bound}
\newconstant{c:annulus_bound}
Let $\alpha > 1$. 
There is a constant $\uc{c:annulus_bound} = \uc{c:annulus_bound}(\alpha) > 0$
such that for every $l_1$, $l_2$ with $l_2 - l_1 \geq 2$ and $l_1 \geq 1$
we have
\begin{equation}
\label{eq:annulus_bound}
\PP (B(l_1) \leftrightarrow_1 \partial B(l_2))
    \le 1 - \exp\bigl[ - u \uc{c:annulus_bound} l_1
        \cdot (l_2 - l_1)^{1-\alpha}\bigr].
\end{equation}
\end{lema}

\begin{proof}
See~\cite[Lemma~6.1]{teixeira2017ellipses}.
The estimate for $\mu(\Gamma_{12})$ implies~\eqref{eq:annulus_bound}.
\end{proof}

\medskip
\noindent
\textbf{Coupling long-range with continuous model.}
There is a canonical coupling mentioned in~\cite{newman1986one}
and used in~\cite{biskup2019sharp} between long-range percolation
model and a Poisson Point Process $\xi$ on $\RR^d \times \RR^d$ with
intensity measure
\begin{equation}
\label{eq:cont_long_rang_inte}
\mu_{\beta, s} := \frac{\beta}{|x-y|^{s}} \,\mathrm{d} x \,\mathrm{d} y.
\end{equation}
We may interpret each point $(x, y) \in \xi$ as giving rise to a segment connecting $x$ and $y$.
This coupling is useful to make the renormalization scaling more transparent. 
In fact, for $a > 0$ if  ${\xi' := \{(ax,ay);\; (x,y) \in \xi\}}$ then the intensities of
$\xi$ and $\xi'$ are related by
\begin{equation}
\label{eq:scaling_mu}
\mu_{\beta}(\mathrm{d} x' \,\mathrm{d} y')
    = \frac{\beta}{|x'-y'|^s} \mathrm{d} x' \,\mathrm{d} y'
    = \frac{\beta}{a^s |x-y|^s} \, a^d \mathrm{d} x \, a^d \mathrm{d} y
    = \mu_{\beta a^{2d-s}}(\mathrm{d} x \,\mathrm{d} y).
\end{equation}
This scaling property is behind the highway property in case $s \in (d, 2d)$,
since the intensity appearing on the right-hand side tends to infinity as $a$
grows. Also notice that when $s=2d$ the model is scale-invariant and there is
no hope that a similar property is satisfied in that case.

For disjoint regions $A_1$ and $A_2$ we write $A_1 \sim_{\xi} A_2$ and say that
$A_1$ and $A_2$ are connected if there is $(\tilde{x}, \tilde{y}) \in (A_1
\times A_2) \cap \xi$.  Take $B:=[-1/2, 1/2)^{d}$ and for $x \in \ZZ^{d}$
consider $B_x = x + B$.  We say that two sites $x \neq y \in \ZZ^{d}$ are
$\xi$-connected if $B_x \sim_{\xi} B_y$ and denote this event by $x \sim_{\xi}
y$.

Lemma~\ref{lema:continuous_coupling} below  yields estimates on the probability
of connecting two distant boxes and shows that this coupling indeed produces a
long-range percolation model. 

\begin{lema}[Connecting boxes]
\label{lema:continuous_coupling}
Let $B(l) = [-l/2,l/2]^{d}$ and $z \in \RR^{d}$ and $B_z(l)= z+B(l)$.
We have that $\PP(B(l) \sim_{\xi} B_z(l)) = 1$ if and only if 
$B(l) \cap B_z(l) \neq \varnothing$. Moreover, we have
\begin{equation}
\label{eq:continuous_coupling}
\PP(B(l) \sim_{\xi} B_z(l)) \sim \beta l^{2d} |z|^{-s}
\quad \textup{as $z \to \infty$}.
\end{equation}
\end{lema}

\begin{proof}
We begin by noticing that for $x \in B(l)$ and $y \in B_z(l)$
we have that
\begin{equation*}
|x-y|
    \geq |-z + x + (z-y)|_{\infty}
    \geq |z|_{\infty} - l.
\end{equation*}
Thus, when $|z|_{\infty} > l$ we  can write
\begin{align*}
\PP\bigl(B(l) \sim_{\xi} B_z(l)\bigr)
    &= 1 - \exp\Bigl[
        - \beta \int_{B(l)\times B_z(l)}
            \hspace{-6mm}|x-y|^{-s} \,\mathrm{d}x \,\mathrm{d}y
        \Bigr] \\
    &\le 1 - \exp\Bigl[ - \beta(|z|_{\infty} - l)^{-s} \cdot l^{2d} \Bigr]
    < 1
\end{align*}
and as $|z| \to \infty$ we have
$|x-y|^{-s} = |z|^{-s} \bigl(1 + O(|z|^{-1})\bigr)$, implying
\begin{equation*}
\PP(B(l) \sim_{\xi} B_z(l))
    = 1 - \exp\Bigl[
        - |z|^{-s} \bigl(\beta + O(|z|^{-1})\bigr) \cdot l^{2d}
        \Bigr]
    \sim \beta l^{2d} |z|^{-s}
\end{equation*}
with implied constants depending on $d$, $s$ and $l$. Also, when
$|z|_{\infty}=l$ one can verify that $\PP(B(l) \sim_{\xi} B_z(l))=1$.
The fact that boxes $B(l)$ and $B_z(l)$ must share at least a corner will
imply the integral diverges for $s \in (d, 2d)$.
\end{proof}

Also, if we restrict our intensity measure to only allow for segments whose
lengths are larger than some fixed value, say $\mu_{\beta, s} :=
{\beta}{|x-y|^{-s}} \I_{|x-y|> \kappa}\,\mathrm{d} x \,\mathrm{d} y$ we
get a model in which nearest neighbors are no longer connected with probability
1, but that has the same behavior on long edges.

\medskip
\noindent
\textbf{Change of variables and ellipses model.}
Now, let us restrict ourselves to the case $d=2$. 
Here we use a change of variables to verify that the PPP's with intensity
measures \eqref{eq:cont_long_rang_inte} and~\eqref{eq:ellipses_measure}
may be viewed as reparametrizations of each other.

Instead of parametrizing a line segment in $\mathbb{R}^2$ specifying its
endpoints $x$ and $y$, we can use its middle point $z = (z_1, z_2)$, its radius
$R$ and the angle if forms with a given direction, $V$.  This change of
variables is given by 
$\Psi: \RR^2 \times \RR^{+} \times [-\tfrac{\pi}{2}, \tfrac{\pi}{2})
\to \RR^{4}$ such that
\begin{equation*}
\Psi(z_1, z_2, R, V)
    = (z_1 + R \cos V, z_2 + R \sin V, z_1 - R \cos V, z_2 - R \sin V).
\end{equation*}
It is straightforward to check that the Jacobian matrix $J$ satisfies
\begin{equation*}
J = 
{\footnotesize
\begin{bmatrix}
\begin{array}{rrrr}
    1 & 0 &   \cos V & - R \sin V \\
    0 & 1 &   \sin V &   R \cos V \\
    1 & 0 & - \cos V &   R \sin V \\
    0 & 1 & - \cos V & - R \cos V
\end{array}
\end{bmatrix}
}
\quad \text{and} \quad
\det J = 4R.
\end{equation*}

Therefore, for any measurable $A \subset
\RR^4$
\begin{align*}
\int_{A} \frac{\beta}{|x-y|^s} \,\mathrm{d} x \,\mathrm{d} y
    &= \int_{\Psi^{-1}(A)} \frac{\beta}{(2R)^s} \cdot (4R)
        \,\mathrm{d} z \,\mathrm{d} R \,\mathrm{d} V \\
    &= \int_{\Psi^{-1}(A)} \frac{4 \beta}{2^s} \cdot R^{1-s}
        \,\mathrm{d} z \,\mathrm{d} R \,\mathrm{d} V.
\end{align*}
The usual parametrization of ellipses percolation is based on measure
\begin{equation*}
    u \,\mathrm{d} z \otimes
    \alpha R^{-(1+\alpha)} \,\mathrm{d} R \otimes
    \frac{\mathrm{d} V}{\pi}.
\end{equation*}
Comparing these measures, we obtain that we can
relate parameters $\beta, s$ used in the endpoint parametrization with the
$u, \alpha$ parametrization of ellipses model, which leads
to the relations in~\eqref{eq:parameters_relation}:
\begin{equation*}
s= 2+ \alpha
\quad \text{and} \quad
u = \frac{\pi \beta}{\alpha 2^{\alpha}}.
\end{equation*}
Using relation~\eqref{eq:parameters_relation} we see that
Lemma~\ref{lema:continuous_coupling} can also be used to estimate connection
probabilities on ellipses model.

\medskip
\noindent
\textbf{Hierarchical construction.} Consider long-range percolation on
$\ZZ^{d}$ with parameters $\beta$ and $s$. The highway property
in~\eqref{eq:highway} ensures that, for fixed $\gamma \in (s/2d,1)$ and
$|x-y| =: N$ large, there is an open edge connecting points in
neighborhoods of size $N^\gamma$ around $x$ and $y$ with high probability.
This idea can be iterated to build what we call a \textit{hierarchy}, see
Definition~\ref{defi:hierarchy} and Figure~\ref{fig:hierarchy}.  When a
hierarchy exists the problem of finding a path from $x$ to $y$ can be replaced
by finding paths between well-separated pairs of points that are however much
closer than the original pair $(x,y)$.  This construction, introduced
by~\cite{biskup2004scaling}, is reproduced below. 

We use $\sigma \in \{0,1\}^{k}$ to encode the leaves of a binary tree of depth
$k$, by considering that $\varnothing$ is the root vertex, and $0$ and $1$
denote the left and right children of $\varnothing$, respectively.  We append
digits to the right of a word $\sigma \in \{0,1\}^k$ in order to create longer
words, e.g., $\sigma1 \in \{0,1\}^{k+1}$ is the word that encodes the right
child of $\sigma$.

\begin{defi}[Hierarchy]
\label{defi:hierarchy}
For $n \geq 1$ and $x, y \in \ZZ^{d}$ we say that a collection
\begin{equation}
\cH_n(x, y)
    = \{(z_{\sigma});
        \sigma \in \big\{0,1\}^{k}, 1 \le k \le n,
        z_{\sigma} \in \ZZ^{d}\big\}
\end{equation}
is a \textit{hierarchy of depth $n$} if 
\begin{enumerate}
\setlength{\itemsep}{1pt}
\setlength{\parskip}{0pt}
\item $z_0 = x$ and $z_1 = y$.
\item For all $0 \le k \le n-2$ and all $\sigma \in \{0,1\}^{k}$$, z_{\sigma 00} = z_{\sigma 0}$ and $z_{\sigma 11} = z_{\sigma 1}$.
\item For all $0 \le k \le n-2$ and all $\sigma \in \{0,1\}^{k}$ with
    $z_{\sigma 01} \neq z_{\sigma 10}$ the edge
    $(z_{\sigma 01}, z_{\sigma 10})$ is open.
\item Each edge $(z_{\sigma 01}, z_{\sigma 10})$ as in 3.\ appears
    exactly once in $\cH_k(x, y)$.
\end{enumerate}
\end{defi}

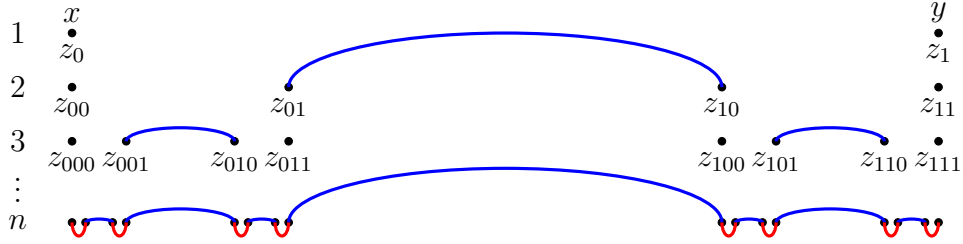
\begin{figure}
\centering
\begin{tikzpicture}[scale=.18,
    dot/.style={
        draw,circle,minimum size=1mm,inner sep=0pt,outer
        sep=0pt,fill=black},
    highway/.style={blue, very thick},
    gap/.style={red, very thick, out=-90, in=-90}
    ]
    \node at (-4, 0) {$1$};
    \coordinate [dot] (0) at ( 0,0)
        node at (0) [below] {$z_{0}$}
        node at (0) [above] {$x$};
    \coordinate [dot] (1) at (64,0)
        node at (1) [below] {$z_{1}$}
        node at (1) [above] {$y$};
    \node at (-4, -4) {$2$};
    \foreach \zlabel/\pos in {00/0, 01/16, 10/48, 11/64}{
        \coordinate [dot] (\zlabel) at ( \pos,-4)
            node at (\zlabel) [below] {$z_{\zlabel}$};
    }
    \node at (-4,-8) {$3$};
    \foreach \zlabel/\pos in {
        000/0, 011/16, 100/48, 111/64,
        001/4, 010/12, 101/52, 110/60
        }{
        \coordinate [dot] (\zlabel) at ( \pos,-8)
            node at (\zlabel) [below] {$z_{\zlabel}$};
    }
    \node at (-4,-11) {$\vdots$};
    \node at (-4,-14) {$n$};
    \foreach \pos in {
        0, 16, 48, 64,
        4, 12, 52, 60,
        1, 15, 49, 63,
        3, 13, 51, 61
        }{
        \node [dot] at ( \pos,-14) {};
    }

    \foreach \ipos/\fpos in {
        1/3, 4/12, 13/15, 16/48, 49/51, 52/60, 61/63
    }{
        \draw[highway] (\ipos, -14)
            arc (180:0:{(\fpos - \ipos)/2} and {(\fpos - \ipos)/8});
    }
    \draw[highway] (01) arc (180:0:16 and 4);
    \draw[highway] (001) arc (180:0:4 and 1);
    \draw[highway] (101) arc (180:0:4 and 1);

    \foreach \ipos in {
        0, 3, 12, 15, 48, 51, 60, 63
    }{
        \draw[gap] (\ipos, -14) arc (-180:0:.5 and 1);
    }
\end{tikzpicture}
\caption{Hierarchy $\cH_n(x, y)$ provides a collection of highways connecting
all pairs $(z_{\sigma 01}, z_{\sigma 10})$ with $\sigma \in \{0,1\}^{n-2}$. To
ensure $x$ is connected to $y$ it suffices to connect the remaining $2^{n-1}$
gaps, that are either of the form $(z_{\sigma 00}, z_{\sigma 01})$ or
$(z_{\sigma 10}, z_{\sigma 11})$.}
\label{fig:hierarchy}
\end{figure}

Note that the definition of a hierarchy does not take into account the distances between the points $z_{\sigma}$.

It will be useful to think of hierarchies as being constructed successively.
In view of the computation in \eqref{eq:highway}, in the first step we may try
to link a pair of sites $z_{01}$ and $z_{10}$ that belong to neighborhoods of
size roughly $N^{\gamma}$ around $z_0$ and $z_1$, respectively (recall that we
are assuming $\gamma \in (s/2d, 1)$ as in the paragraph
above~\eqref{eq:highway}).  Having succeeded to do so in the first $k$ steps,
for each $\sigma \in \{0,1\}^k$ we try to link $z_{\sigma 01}$ and $z_{\sigma
10}$ belonging to neighborhoods of size roughly $ N^{\gamma^k}$ around
$z_{\sigma 0}$ and $z_{\sigma 1}$ respectively. Ideally, when we reach depth
$n$ we will be left with $2^{n-1}$ gaps which are pairs of sites of type
$(z_{\sigma 00}, z_{\sigma 01})$ or $(z_{\sigma 10}, z_{\sigma 11})$ with sites
in each pair at a distance of order $N^{\gamma^{n}}$.  The reader may
consult Figure \ref{fig:hierarchy} for an illustration of this iterative
procedure.  Note that, by the discrete nature of the long-range model the
procedure   cannot be iterated indefinitely.

The above discussion motivates the definition of the event $\cB_n(x, y)$ that
there is a hierarchy $\cH_n(x,y)$ of depth $n$ satisfying that, for all $0 \le
k \le n-2$ and all $\sigma \in \{0,1\}^{k}$
\begin{equation}
\label{eq:cB_n_definition}
|z_{\sigma 01} - z_{\sigma 00}|_{\infty}
\quad \text{and} \quad
|z_{\sigma 10} - z_{\sigma 11}|_{\infty}
\quad \text{belong to $\Bigl[\tfrac{1}{2} N_{k+1}, N_{k+1}\Bigr]$,}
\end{equation}
where $N_k:= N^{\gamma^k}$. The following lemma is a simplified version
of~\cite[Lemmas~4.2 and 4.3]{biskup2004scaling} and provides appropriate
choices of parameters so that the above idealized picture is achieved with high
probability.
\begin{lema}[Hierarchy]
\label{lema:hierarchy}
Fix $\varepsilon > 0$ and $\gamma \in \bigl(\frac{s}{2d}, 1\bigr)$. 
For $x, y \in \ZZ^{d}$ and $N := |x-y|$, let $n \in \NN$ be the greatest
positive integer such that 
\begin{equation}
n \log (1/\gamma) \le \log^{(2)} N - \varepsilon \log^{(3)} N.
\label{eq:relation_n_N}
\end{equation}
There is $N'(\varepsilon, \gamma, d)$ and
$b = b(d) \in (0,1)$ such that if $N \geq N'$ then for any hierarchy
$\cH_n (x,y)$ satisfying~\eqref{eq:cB_n_definition} we have
\begin{equation}
\label{eq:highway_dist_bounds}
\forall\, 0 \le k \le n-2, \text{ and } \sigma \in \{0,1\}^{k},\,\,\,
|z_{\sigma 01} - z_{\sigma 10}| \in [b N_{k}, b^{-1} N_{k}].
\end{equation}
Moreover, there is a positive constant $c = c(\beta, d, s)$ such that
\begin{equation}
\label{eq:prob_cB_n_bound}
\PP\bigl(\cB_n(x,y)^{\comp}\bigr)
    \le 2^{n-1} \cdot e^{ - c N_n^{(2d\gamma - s)} }
    \le \exp\bigl[ - c e^{(2d\gamma - s) (\log^{(2)} N)^{\varepsilon}} \bigr].
\end{equation}
\end{lema}

\begin{remark}
Let $\Delta' := \frac{\log 2}{\log (1/\gamma)}$.
The definition of $n$ yields:
\begin{equation}
\label{eq:n_defi_consequences}
2^{n}
    \le (\log N)^{\Delta'}
\quad \text{and} \quad
e^{(\log^{(2)} N)^{\varepsilon}}
    \le N_n \le e^{(1/\gamma)(\log^{(2)} N)^{\varepsilon}}.
\end{equation}

In fact, by the definition of $n$ and $N_n$ we can write
\begin{align*}
N_n
    &:= N^{\gamma^{n}}
    = \exp\bigl[ \exp[ \log^{(2)} N + n \log \gamma ] \bigr], \\
\text{and also} \qquad
\varepsilon \log^{(3)} N
    &\le \log^{(2)} N + n \log \gamma
    \le \varepsilon \log^{(3)} N + \log (1/\gamma).
\end{align*}
Therefore,
\begin{equation*}
e^{(\log^{(2)} N)^{\varepsilon}}
    =   e^{\exp[ \varepsilon \log^{(3)} N ]}
    \le N_n
    \le e^{\exp[ \varepsilon \log^{(3)} N + \log (1/\gamma) ]}
    =   e^{(1/\gamma)(\log^{(2)} N)^{\varepsilon}}.
\end{equation*}
The other inequality in~\eqref{eq:n_defi_consequences} follows from
$n \log (1/\gamma) \le \log^{(2)} N$.

\end{remark}

\begin{proof}[Proof of Lemma \ref{lema:hierarchy}]
By~\eqref{eq:n_defi_consequences}, we have that scales $N_k$ are 
well-separated meaning that, for $0 \le k \le n-1$,
\begin{equation}
\label{eq:scale_ratio}
\frac{N_k}{N_{k+1}}
    = N_k^{1 - \gamma}
    \geq N_{n}^{1-\gamma}
    \geq e^{(1-\gamma) (\log^{(2)} N)^{\varepsilon}}.
\end{equation}

Let $0 \le k \le n-2$ and $\sigma \in \{0,1\}^{k}$.
By~\eqref{eq:cB_n_definition} we have that
\begin{equation*}
|z_{\sigma 0} - z_{\sigma 1}|_{\infty} = \Theta(N_k)
\quad \text{and} \quad
|z_{\sigma 0} - z_{\sigma 01}|_{\infty},
|z_{\sigma 1} - z_{\sigma 10}|_{\infty} = \Theta(N_{k+1})
\end{equation*}
and thus the triangular inequality and \eqref{eq:scale_ratio} imply
\begin{equation*}
|z_{\sigma 01} - z_{\sigma 10}|_{\infty}
    = \Theta(N_{k}) + 2\Theta(N_{k+1})
    = \Theta(N_k),
\end{equation*}
and~\eqref{eq:highway_dist_bounds} follows. 

To obtain~\eqref{eq:prob_cB_n_bound} we partition $\cB_n^{\comp}$ according to
the first depth $k$ at which we fail to find highways. For that value of
$k$ the event $\mathcal{B}_{k-1}$ occurs and there is a hierarchy
$\mathcal{H}_{k-1}$ satisfying \eqref{eq:cB_n_definition}.
Let us fix a gap in $\cH_{k-1}$, which is either the form
$(z_{\sigma 00}, z_{\sigma 01})$ or $(z_{\sigma 10}, z_{\sigma 11})$ with
$\sigma \in \{0,1\}^{k-3}$. For the corresponding pair of neighborhoods, say
\begin{equation*}
\{z' \colon \tfrac{1}{2} N_{k-1} \le |z_{\sigma 00}-z'|_{\infty} \le
N_{k-1}\}
\quad \text{and} \quad
\{z' \colon \tfrac{1}{2} N_{k-1} \le |z_{\sigma
01}-z'|_{\infty} \le N_{k-1}\},
\end{equation*}
none of the edges between these neighborhoods must be open.
By~\eqref{eq:cB_n_definition} these neighborhoods are centered at sites
whose distance belongs to $\bigl[\tfrac{1}{2}N_{k-2}, N_{k-2}\bigr]$.
Moreover, each neighborhood has $c N_{k-1}^{d}$ vertices.
A straightforward adaptation of the argument leading
to~\eqref{eq:highway} applied to scale $N_{k-2}$ shows that the probability
of not finding an open edge linking a fixed pair of neighborhoods is
bounded above by
\begin{align*}
    & \exp\bigl[ - \beta c N_{k-1}^{2d} N^{-s}_{k-2} \bigr]
     \le \exp\bigl[ - c N_{n}^{2d\gamma - s} \bigr],
\end{align*}
where $c = c(\beta, d, s) > 0$. 
Since there are $2^{k-2}$ pairs of neighborhoods we have:
\begin{equation*}
\PP(\cB_n^{\comp})
    =   \sum_{k=2}^{n} \PP(\cB_{k-1} \cap \cB_{k}^{\comp})
    \le \sum_{k=2}^{n} 2^{k-2}
        e^{ - c N_{n}^{2d\gamma - s} }
    \le 2^{n-1} \cdot e^{ - c N_{n}^{2d\gamma - s} }.
\end{equation*}
The last bound in~\eqref{eq:prob_cB_n_bound} follows
from~\eqref{eq:n_defi_consequences}.
\end{proof}

\section{Bounding the Euclidean distance}
\label{sec:bounding_euclidean_distance}

Using the hierarchical construction from Biskup~\cite{biskup2004scaling} we
obtain a collection of highways that provides the main contribution for finding
open paths between two distant sites.  However, the remaining gaps must still
be connected in order to find an open path between the original two points that
actually uses these highways.  In~\cite{biskup2004scaling} this is accomplished
by requiring that the vertices $z_{\sigma}$ in $\cH_n(x,y)$ belong to
sufficiently large but local clusters. For our model we take a similar but
different strategy.

The idea is to make a hybrid approach, considering a renormalized lattice to
define a site-bond percolation model. The bond percolation part will be
coupled to a long-range percolation model. Independently of this bond
percolation part we define a site percolation that will be used to glue
together these highways, using an idea from Antal and
Pisztora~\cite{antal1996chemical}.

\subsection{Renormalization scheme}

We begin describing the renormalization procedure.  Partition $\ZZ^2$ into a
collection of boxes $(B_x;\; x \in \ZZ^2)$ where $B_x = B_x(K) := Kx +
[-K/2,K/2]^{2}$.  The exact choice of $K$ will depend on the parameters $u,
\alpha$ of the model and is deferred till
Lemma~\ref{lema:renormalized_properties}.  Each box $B_x$ is assigned an
\textit{enlarged box} and a \textit{core}, defined respectively as
\begin{equation*}
    B'_x := Kx + \bigl[-\tfrac{9K}{10}, \tfrac{9K}{10}\bigr]^{2}
    \quad \text{and} \quad
    B''_x := Kx + \bigl[-\tfrac{K}{10}, \tfrac{K}{10}\bigr]^{2}.
\end{equation*}
We say that the box of the origin $B_o$ is \textit{good} if the event
\begin{equation}
\label{eq:event_for_good_box}
\Bigl\{
    \bigl[\tfrac{5K}{10}, \tfrac{7K}{10}\bigr] \times
    \bigl[-\tfrac{9K}{10},-\tfrac{7K}{10}\bigr]
\sim_{\xi}
    \bigl[\tfrac{5K}{10}, \tfrac{7K}{10}\bigr] \times
    \bigl[ \tfrac{7K}{10}, \tfrac{9K}{10}\bigr]
\Bigr\}
\end{equation}
occurs, as well as the three similar events resulting from
\eqref{eq:event_for_good_box} by rotations by $\tfrac{\pi}{2}$, $\pi$ and
$\tfrac{3\pi}{2}$ around the origin, see Figure~\ref{fig:good_box}.  The events
$\{\text{$B_x$ is good}\}$ are defined analogously; in words, a box $B_x$ is
good if it is enclosed by a well-positioned circuit of overlapping ellipses
contained in its enlarged box, see Figure~\ref{fig:good_box}. We also say that
a site $x$ is good if its respective renormalized box, $B_x$, is good.
If $B_x$ is good, we denote by $O_x$ a circuit of ellipses that realizes
such event, chosen according to some predetermined rule.

Our site-bond percolation model is defined via the following configurations:
\begin{description}
\item[\textbf{Site percolation:}] 
$\bigl(\omega_{x}\bigr)
    := \bigl(\I\{\text{$B_x$ is good}\}; x \in \ZZ^{2} \bigr)$;
\item[\textbf{Bond percolation:}]
$\bigl(\omega_{xy}\bigr)
    := \bigl(\I\{B''_x \sim_{\xi} B''_y\};\;
        x,y \in \ZZ^{2}, x \neq y \bigr)$.
\end{description}

It follows from our construction that $\bigl(\omega_{x}\bigr)$ and
$\bigl(\omega_{xy}\bigr)$ are independent processes since they are defined in
terms of the PPP $\xi$ restricted to disjoint regions of $\RR^4$.  For the same
reason, $(\omega_x)$ is an independent (Bernoulli) site percolation process on
$\mathbb{Z}^2$.

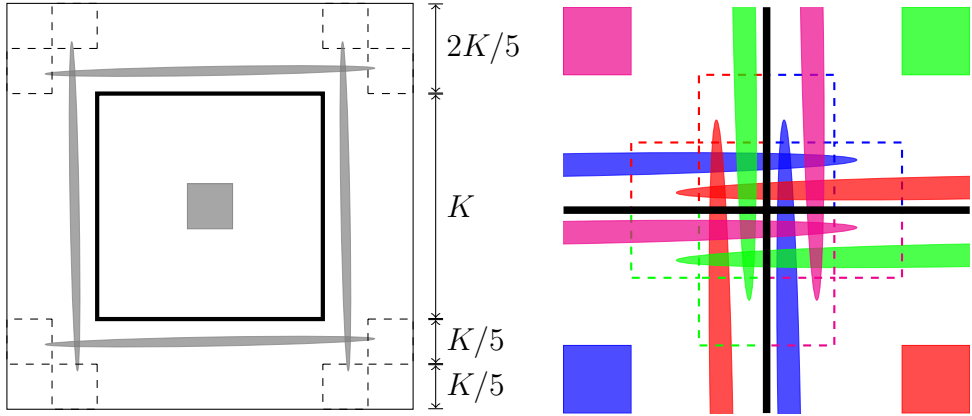
\begin{figure}
\centering
\begin{tikzpicture}[scale=0.3]
    \draw[ultra thick] (-5,-5) rectangle (5,5);
    \draw (-9,-9) rectangle (9,9);
    \filldraw[gray, opacity=.7] (-1,-1) rectangle (1,1);
    \foreach \x in {0, 90, 180, 270}{
        \begin{scope}[rotate=\x]
        \draw[dashed] (5,7) rectangle (7,9)
                      (7,5) rectangle (9,7);
        \filldraw[gray, opacity=.7,rotate around={ 1:(0,6)}]
            (0,6) circle (7.3 and .2);
        \end{scope}
    }
    \draw[|<->|] (10,-9) -- (10,-7) node[midway, right] {$K/5$};
    \draw[|<->|] (10,-7) -- (10,-5) node[midway, right] {$K/5$};
    \draw[|<->|] (10,-5) -- (10, 5) node[midway, right] {$K$};
    \draw[|<->|] (10, 5) -- (10, 9) node[midway, right] {$2K/5$};
\end{tikzpicture}%
\hspace{3mm}
\begin{tikzpicture}[scale=0.3]
\clip (-9,-9) rectangle (9,9);
\foreach \x/\y in {0/blue, 90/red, 180/green, 270/magenta}{
    \begin{scope}[rotate=\x]
    \draw[color=\y,dashed, thick] (0,3) rectangle (3,6)
                  (3,0) rectangle (6,3);
    \filldraw[color=\y,opacity=.7] (-9,-9) rectangle (-6,-6);
    \filldraw[rotate around={ 1:(1,-9)},color=\y,opacity=.7]
        (1,-9) circle (.5 and 13);
    \filldraw[rotate around={ 1:(-9,2)},color=\y,opacity=.7]
        (-9,2) circle (13 and .5);
    \end{scope}
}
\draw[line width=1mm] (-9,0) -- (9,0) (0,-9) -- (0,9);
\end{tikzpicture}
\caption{A renormalized box $B_x$ is good if there is a carefully
    positioned surrounding circuit of ellipses $O_x$ contained in its enlarged box $B'_x$;
    we also highlight its core $B''_x$. 
    On the right, we emphasize that
    $\ast$-neighboring good boxes must have their respective outer
    circuits interlaced. 
    Outer circuits and cores of a same box are shown in
    matching colors to help visualization.}
\label{fig:good_box}
\end{figure}

Our definition of a good box is close to the definition of good boxes used
in~\cite{antal1996chemical}. Essentially, it ensures that in a cluster of good
boxes one is able to move from one box to a neighboring one remaining inside
the covered set. This holds not only when moving along the coordinate
directions (to a box that shares a side) but also when moving diagonally (to a
box that shares a single vertex). We briefly discuss this notion of
connectivity now introducing notation that is very similar to that
of~\cite{antal1996chemical}.

\textbf{On $\ast$-connected sets.}
For $x, y \in \ZZ^{d}$ define
\begin{equation*}
x \sim y \quad \text{if $|x-y|=1$}
\quad \text{and} \quad
x \stackrel{\ast}{\sim} y \quad \text{if $|x-y|_{\infty}=1$}.
\end{equation*}
Given a configuration ${\omega \in \{0,1\}^{\ZZ^{d}}}$, we say that a site $x
\in \mathbb{Z}^d$ is good if $\omega_x=1$.  Otherwise $x$ is said bad.  Denote
by $\badC_x$ the bad cluster (with respect to $\stackrel{\ast}{\sim}$) containing $x$.
We use the convention that $\badC_x = \varnothing$ if $x$ is good.  
For a finite subset $\Lambda \subset \ZZ^{d}$
define its \textit{outer} and \textit{inner} boundaries by
\begin{equation*}
\partial^{o} \Lambda
    := \{x \in \Lambda^{c};\; \exists y \in \Lambda, x \sim y\}
\quad \text{and} \quad
\partial^{i} \Lambda
    := \{x \in \Lambda;\; \exists y \in \Lambda^{c}, x \sim y\},
\end{equation*}
respectively.
We use the convention that $\partial^{o} \badC_x = \{x\}$ when $x$ is good.

For $\Lambda$ finite, its complementary set $\Lambda^{c}$ contains a finite
number of connected components $\Lambda_1, \ldots, \Lambda_k$.  Exactly one of
them, say $\Lambda_1$ is infinite; the other ones, if any, are called
\textit{holes}.  When holes exist, we define $\hLambda := \Lambda \cup
\Lambda_2 \cup \ldots \cup \Lambda_k$ which may be regarded as the result of
filling all holes in $\Lambda$.  We also define the \textit{external outer
boundary} and the \textit{external inner boundary} of $\Lambda$ respectively as
\begin{equation*}
\partial^{o}_{e} \Lambda := \partial^{o} \hLambda
\quad \text{and} \quad
\partial^{i}_{e} \Lambda := \partial^{i} \hLambda.
\end{equation*}
An important topological fact is given
by~\cite[Statement (3.35)]{antal1996chemical}:
\begin{display}
\label{eq:ast_connect_property}
for any finite $\Lambda$ that is $\ast$-connected we have that\\
$\partial^{o}_{e} \Lambda$ and $\partial^{i}_{e} \Lambda$
are both $\ast$-connected.
\end{display}
This is important because whenever we find a region composed of bad sites, we
can  contour that bad region using its exterior boundary of good sites. 

\medskip
\noindent
\textbf{On the renormalized model.}
Recall that the PPP $\xi$ can be parametrized using either $(s, \beta)$ or
$(u, \alpha)$, by~\eqref{eq:parameters_relation}.

\begin{lema}
\label{lema:renormalized_properties}
The following properties hold for the renormalized model:
\begin{itemize}
\item[\textbf{\textup{P1}}.] Fix $p \in (0,1)$ and $\beta_0 > 0$.
    There is $K(u, \alpha, p, \beta_0)$ large enough such that
    $(\omega_x)$ dominates independent site percolation with parameter $p$ and
    $(\omega_{xy})$ dominates long-range percolation with
    $p_{xy} = 1 - e^{-\beta_0 |x-y|^{-s}}$.
\item[\textbf{\textup{P2}}.] If $W$ is a $\ast$-connected set of good sites then all
    the surrounding circuits $O_x$, $x\in W$  are contained in the same connected component of ellipses.
\item[\textbf{\textup{P3}}.]   If $\badC_x$ is finite, then
    $\partial^{o}_{e} \badC_x$ is a $\ast$-connected set of good sites.
\end{itemize}
\end{lema}

\begin{proof}
Property \textbf{P2} is a straightforward geometric consequence of the
definition  (see Figure~\ref{fig:good_box}) and Property \textbf{P3}
follows from~\eqref{eq:ast_connect_property}.

We now prove Property \textbf{P1}.  Denote by $A_0$ the event
in~\eqref{eq:event_for_good_box} and by $A_i$, for $1 \le i \le 3$ the
three similar events resulting from \eqref{eq:event_for_good_box} by
rotations by $\tfrac{\pi}{2}$, $\pi$ and $\tfrac{3\pi}{2}$ around the
origin, respectively.  Since $\xi$ is invariant with respect to
translations and rotations, any $A_i$ has probability
\begin{equation*}
\PP(A_i)
    = \PP\bigl([-\tfrac{K}{10}, \tfrac{K}{10}]^2 \sim_{\xi}
        (\tfrac{8K}{5}, 0) + [-\tfrac{K}{10}, \tfrac{K}{10}]^2\bigr)
    \geq 1 - e^{
            - \beta \, \smash{(\max\limits_{x,y}} {|x-y|)}^{-s} \cdot
            {(\frac{K}{5})}^{4}
        }
\end{equation*}
where the maximum runs over all points $x$ and $y$ in the first and second
boxes, respectively. The maximum is a constant multiple of $K$, so we can write
\begin{equation*}
\PP(A_i)
    \geq 1 - e^{ - c K^{4-s} } \to 1
    \quad \text{as $K \to \infty$}
\end{equation*}
for some constant $c = c(\beta, s) > 0$, since we are assuming
$s \in (3, 4)$. Then, FKG inequality implies $\PP(\text{$B_o$ is good})$ also
tends to $1$.
 
For the probability of an edge being open, we notice that
$B''_x \sim_{\xi} B''_y$ is a scaling by $K/5$ of event 
\begin{equation}
\label{eq:rescaled_crossing}
\Bigl\{ \bigl[-\tfrac{1}{2}, \tfrac{1}{2}\bigr]^{2} \sim_{\xi}
    5(y-x) + \bigl[-\tfrac{1}{2}, \tfrac{1}{2}\bigr]^{2} \Bigr\}.
\end{equation}

By~\eqref{eq:scaling_mu} we can relate the probability of $B''_x \sim_{\xi}
B''_y$ with that of the event in \eqref{eq:rescaled_crossing} under a
rescaled long-range model whose intensity can be made as high as we want by
increasing $K$.  This completes the proof of Property \textbf{P1}.
\end{proof}

\subsection{Gluing highways}
\label{sub:gluing_highways}

Given two fixed sites $x, y \in \RR^2$, Lemma~\ref{lema:hierarchy} roughly
states that, for the long-range model in the renormalized lattice, hierarchies
exist with very high probability. On the event $x \leftrightarrow y$ we want
to use the highway structure entailed by one of these hierarchies in order to
find a path that connects $x$ to $y$ efficiently.

For $z \in \mathbb{R}^2$, let $a(z) \in \ZZ^{2}$ be the unique site in
$\ZZ^{2}$ such that $z \in Ka(z) + [-K/2,K/2)^{2}$.  The distance between the
original points $x$ and $y$ and the distance between their respective
counterparts $a(x)$ and $a(y)$ in the renormalized lattice can be compared as
\begin{align}
|x-y|
    &= \bigl|(x - Ka(x)) - (y - Ka(y)) + K(a(x) - a(y))\bigr| \nonumber\\
\label{eq:relation_points_sites}
    &= K \cdot \bigl|a(x) - a(y)\bigr| + KO(1).
\end{align}
Here the $L_2$-norm could be replaced by any other norm on $\RR^2$.
Lemma~\ref{lema:hierarchy} implies that $\cB_n(a(x), a(y))$
has probability close to 1 (provided that $N$ is large and $n$ satisfies
\eqref{eq:relation_n_N}).  Conditional on $\cB_n$, we can find a collection of
sites $\bigl(z_{\sigma}; \sigma \in \{0,1\}^{n}\bigr)$ together with the
endowed highway structure connecting some of them.  If there is more than one
choice, just pick one of them according to a predetermined rule.

We still have to ensure that all the remaining gaps, that is, all the $2^{n-1}$
edges of type $(z_{\sigma 00}, z_{\sigma 01})$ or $(z_{\sigma 10}, z_{\sigma
11})$ for $\sigma \in \{0,1\}^{n-2}$ are connected with high probability.
This can be done with the aid of an argument from~\cite{antal1996chemical}.

We can specify each gap uniquely as $(z_{\sigma 0}, z_{\sigma 1})$, with
$\sigma \in \{0,1\}^{n-1}$.  Writing $\tilde{N} := |a(x) - a(y)|$ and
$\tilde{N}_k = \tilde{N}^{\gamma^{k}}$, on the event $\cB_n(a(x), a(y))$ we
have 
\begin{equation}
\label{eq:dist_gaps}
|z_{\sigma 1} - z_{\sigma 0}|_{\infty}
    \in \Bigl[\tfrac{1}{2} \tilde{N}_{n-1}, \tilde{N}_{n-1}\Bigr]
\quad \text{for every $\sigma \in \{0,1\}^{n-1}$,}
\end{equation}
see \eqref{eq:cB_n_definition}.

Moreover, \eqref{eq:highway_dist_bounds} 
guarantees that the highways of the form $(z_{\sigma 01}, z_{\sigma 10})$ with
$\sigma \in \{0,1\}^{k}$ and $0\leq k \leq n-2$ have length
$|z_{\sigma 01} - z_{\sigma 10}| = \Theta(\tilde{N}_{k})$.
From the point of view of our original ellipses model, a highway connecting 
$z_{\sigma 01}$ and $z_{\sigma 10}$ represents an actual ellipse $E_{\sigma}$
that realizes the event $B''_{z_{\sigma 10}} \sim_{\xi} B''_{z_{\sigma 01}}$. Thus,
\newconstant{c:ellipse_length}
\begin{equation}
\label{eq:highway_diam}
\diam (E_{\sigma})
    = K \cdot \bigl(|z_{\sigma 10} - z_{\sigma 01}| + O(1)\bigr)
\end{equation}
and consequently the
\begin{equation}
\label{eq:highway_boxes_touched}
\text{(number of renormalized boxes intersected by $E_{\sigma}$)} 
    \in \bigl[
        \uc{c:ellipse_length} \tilde{N}_{k},
        \uc{c:ellipse_length}^{-1} \tilde{N}_{k}
        \bigr]
\end{equation}
for a positive constant $\uc{c:ellipse_length}$ that will remain fixed from now
on.  Also, the site percolation process $(\omega_x)$ is independent of the
collection $\bigl(z_{\sigma}; \sigma \in \{0,1\}^{n}\bigr)$.

For each gap $(z_{\sigma 0}, z_{\sigma 1})$, with $\sigma \in \{0,1\}^{n-1}$
write $m_{\sigma} := |z_{\sigma 0} - z_{\sigma 1}|_1$ and fix a deterministic
path (according to a predetermined rule) of $m_{\sigma} + 1$ neighboring sites
that realize this distance, meaning
\begin{equation*}
z_{\sigma 0} = z^{(\sigma)}_{0}, z^{(\sigma)}_{1}, \ldots,
z^{(\sigma)}_{m_{\sigma}} = z_{\sigma 1}
\quad \text{
with \quad $z^{(\sigma)}_{j} \sim z^{(\sigma)}_{j+1}$ for $0 \le j < m_{\sigma}$.
}
\end{equation*}
Recall that $\badC_z$ denotes the $\ast$-connected cluster of bad sites
containing $z$ and define $\bbadC(z) := \badC_z \cup \partial^{o}\badC_z$.  We
look at the random subset of $\RR^{2}$ composed by the boxes associated to the
bad clusters of the sites along the path $(z^{(\sigma)}_j)$:
\begin{equation}
\label{eq:defi_W_sigma}
W_{\sigma}
    := \bigcup_{j=0}^{m_{\sigma}} \Bigl(
    \bigcup_{z \in \bbadC(z^{(\sigma)}_j)} B'_{z} \Bigr).
\end{equation}
Denoting by $\# W_{\sigma}$ the number of sites in the renormalized lattice
that one needs to explore to find $W_{\sigma}$, we have:
\newconstant{c:W}
\begin{lema}
\label{lema:control_W_sigma}
For every $a > 0$ there exists $\uc{c:W} = \uc{c:W}(a, p) > 0$
such that
\begin{equation}
\label{eq:control_W_sigma}
\PP \Bigl(
    \cB_{n} \cap \Bigl(\hspace{-3mm}
        \bigcup_{\sigma \in \{0,1\}^{n-1}} \hspace{-3mm}
        \bigl\{\# W_{\sigma} \geq a \tilde{N}_{n-2} \bigr\}
    \Bigr)
\Bigr)
    \le 2^{n-1} \cdot e^{-\uc{c:W} \tilde{N}_{n-2}} \ll 1.
\end{equation}
\end{lema}

\begin{proof}
We have $\# \bbadC_z = 1$ when $z$ is good.
Otherwise, since each site of $\badC_z$ has at most 8 neighbors 
\begin{equation*}
\# W_{\sigma}
    \le \sum_{j=0}^{m_{\sigma}} \# \bbadC(z^{(\sigma)}_j)
    \le 1 + m_{\sigma} + 8 \sum_{j=0}^{m_{\sigma}} \# \badC(z^{(\sigma)}_j).
\end{equation*}

Also, using an argument from~\cite{antal1996chemical} based on a previous
construction in Fontes and Newman~\cite{fontes1993first} (see the proof of
Theorem~4), if $(\tbadC_z)_{z\in \ZZ^{2}}$ is a collection of independent
random subsets of $\ZZ^{2}$ with $\tbadC_z \distr \badC_o$, then
$(\tbadC_z)$ dominates stochastically $(\badC_z)$.
Defining $Y_z := \# \tbadC_z$, we have that $(Y_z, z\in \ZZ^{2})$ are
i.i.d.\ random variables with the same distribution as $\# \badC_o$. 
We have for $Y_j := \smash{Y_{z_j^{(\sigma)}}}$ that
\begin{equation*}
\# W_{\sigma}
    \preccurlyeq 1 + m_{\sigma} +
        8 \sum_{j=0}^{m_{\sigma}} \# \tbadC(z^{(\sigma)}_j)
    \distr 1 + m_{\sigma} + 8 \sum_{j=0}^{m_{\sigma}} Y_j.
\end{equation*}
Notice that $\badC_o$ is a $\ast$-cluster of bad sites in a Bernoulli site
percolation of parameter $p$ and by Lemma~\ref{lema:renormalized_properties}
we can start the construction with $p$ sufficiently close to 1 so that
the probability of a site being bad, $1-p$, is subcritical, and then choose
$K(u, \alpha, p, \beta_0)$ accordingly. Exponential decay of cluster size (see
e.g.~\cite[Theorem~(6.75)]{grimmett1999percolation}) yields
$\psi(p) > 0$ such that $h(p) := \EE\bigl[e^{\psi(p) Y_j}\bigr] < \infty$.
Hence, for any fixed $\sigma \in \{0,1\}^{n-1}$, an application of Markov's Inequality yields
\begin{align*}
\PP \bigl(\cB_{n} \cap \{\# W_{\sigma} \geq a \tilde{N}_{n-2}\} \bigr)
    &\le \PP \Bigl(
        \sum_{j=0}^{m_{\sigma}} Y_j
        \geq \frac{a}{8} \tilde{N}_{n-2} + O(m_{\sigma})
        \Bigr) \\
    &\le \exp\Bigl[- \frac{a\psi(p)}{8} \tilde{N}_{n-2} + O(m_{\sigma})\Bigr] \cdot
        \EE \bigl[ e^{\psi(p) \sum_{j=0}^{m_{\sigma}} Y_j} \bigr] \\
    &= \exp\Bigl[
        - \frac{a\psi(p)}{8} \tilde{N}_{n-2} + O(m_{\sigma}) +
        \log h(p) \cdot m_{\sigma} 
        \Bigr].
\end{align*}
Finally, by~\eqref{eq:dist_gaps} we can write
\begin{equation*}
m_{\sigma}
    = |z_{\sigma 0} - z_{\sigma 1}|_1
    = \Theta(\tilde{N}_{n-1})
    \ll \tilde{N}_{n-2}.
\end{equation*}
The estimate on~\eqref{eq:control_W_sigma} follows from a union bound and the
bounds obtained in~\eqref{eq:n_defi_consequences}.
\end{proof}

Recall the definition of $\uc{c:ellipse_length}$
in~\eqref{eq:highway_boxes_touched} and take
$a = \frac{1}{3}\uc{c:ellipse_length}$ at Lemma~\ref{lema:control_W_sigma}.
Define
\begin{equation}
\cW_n := 
    \hspace{-3mm}
    \smash{\bigcap_{\sigma \in \{0,1\}^{n-1}}} \hspace{-3mm}
    \bigl\{\# W_{\sigma} \le \tfrac{\uc{c:ellipse_length}}{3}
    \tilde{N}_{n-2} \bigr\}.
\end{equation}
We have $\PP(\cB_{n} \cap \cW_{n}^{\comp}) \ll 1$, meaning that with high
probability every $W_{\sigma}$ is too small to contain any highway.  Fix some
$\sigma \in \{0,1\}^{n-1}$.  By \textbf{P2} and \textbf{P3} on
Lemma~\ref{lema:renormalized_properties} we can use the external inner boundary
$\partial^{i}_{e} W_{\sigma}$, a $\ast$-connected set of good boxes, to glue
together the highways that arrive at $z_{\sigma 0}$ and $z_{\sigma 1}$, the
procedure is illustrated in Figure~\ref{fig:gluing_highways}.

\begin{figure}
\centering
\begin{tikzpicture}[scale=1.2,
    dot/.style={
        draw,circle,minimum size=1mm,inner sep=0pt,outer
        sep=0pt,fill=black}
    ]
    \clip (-2,-.5) rectangle (9, 4.5);
    \draw (-2,-.5) grid (9, 4.5);
    \filldraw[gray, opacity=.8]
        (0,0) rectangle ++(1,2)
        (1,0) rectangle ++(1,1)
        (3,0) rectangle ++(1,2)
        (4,2) rectangle ++(3,1)
        (7,1) rectangle ++(1,1);
    \filldraw[gray, opacity=.3]
        (-1,-1) rectangle ++(6,1)
        (-1, 0) rectangle ++(10,3)
        (3,3) rectangle ++(5,1);
    \filldraw[blue, opacity=.7,rotate around={147:(9,2)}] (9, 2) circle (3 and .2);
    \filldraw[blue, opacity=.7,rotate around={-95:(1,5)}] (1, 5) circle (4.5 and .2);
    \node at (0.5,0.3) {$B_{z_{\sigma 1}}$};
    \node at (6.5,3.3) {$B_{z_{\sigma 0}}$};
    \draw[ultra thick] (0,0) grid (7,1) (6,1) grid (7, 4);
\end{tikzpicture}
    \caption{Region $W_{\sigma}$ (light gray) explores
    bad boxes (dark gray) on a deterministic path of boxes (thick lines) and
    $\partial^{i}_{e} W_{\sigma}$ is made of good boxes. When
    $\# W_{\sigma}$ is small, the highways arriving at
    $B_{z_{\sigma 0}}$ and $B_{z_{\sigma 1}}$ can be connected through
    $\partial^{i}_{e} W_{\sigma}$.}
\label{fig:gluing_highways}
\end{figure}
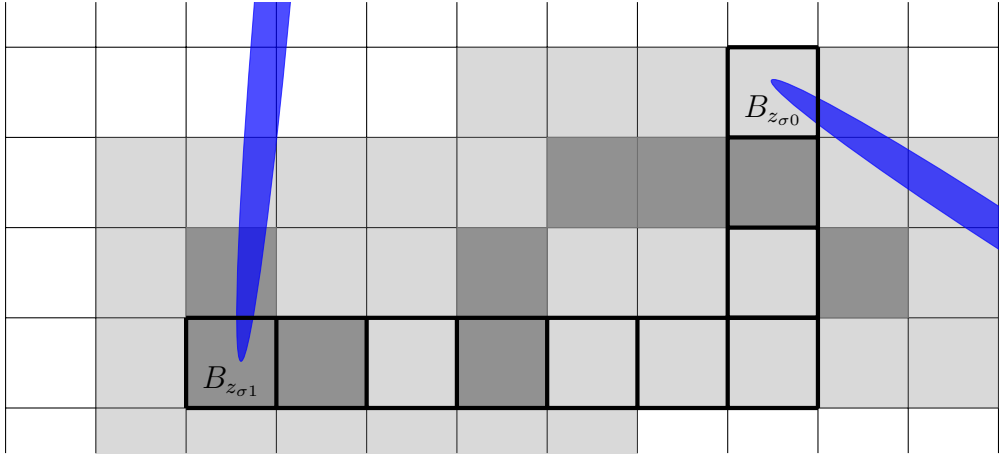

Suppose that we know that $\{x \leftrightarrow y\} \cap \cB_n\bigl(a(x),
a(y)\bigr) \cap \cW_{n}$ has occurred and fix a path $\cP$ connecting $x$ to
$y$.  Although $\cP$ can be arbitrarily long, after the gluing process we can
build a path $\cP'$ from $x$ to $y$ whose length is controlled.  Let
$\underline{0}, \underline{1} \in \{0,1\}^{n-1}$ be the all zeroes and all ones
words, respectively.
\begin{defi}
\label{defi:path_cP_prime}
Path $\cP'$ is defined as follows:
\begin{enumerate}
\setlength{\itemsep}{1pt}
\setlength{\parskip}{0pt}
\item Follow $\cP$ from $x$ till it hits the first outer circuit $O_z$ of a good box $B_z$, with $z \in \partial^{i}_{e} W_{\underline{0}}$.
\item When $\cP'$ first gets to $\partial^{i}_{e} W_{\sigma}$ with
    $\sigma \neq \underline{1}$, use outer circuits to move towards the next highway.
\item When $\cP'$ arrives at a highway, move in a straight line till
    intersecting the next $\partial^{i}_{e} W_{\sigma}$.
\item When $\cP'$ gets to $\partial^{i}_{e} W_{\underline{1}}$, use outer
    circuits to move to the last point of $\cP$ that intersects a circuit
    $O_z$ in $\partial^{i}_{e} W_{\underline{1}}$ and then use
    $\cP$ to move to $y$.
\end{enumerate}
\end{defi}

We have good estimates for the length of path $\cP'$ when moving on highways or
when using outer circuits of some $W_{\sigma}$. 
However, some parts of $\cP'$ could be wiggly (when following along $\cP$) and
that could possibly add a considerable amount to the total length. 
The next lemma allows us to improve the estimate on the length of
a path inside the covered set $\cE$ when we move inside a bounded region.

\begin{lema}[Distance on small scales]
\label{lema:distance_small_scales}
Let $W \subset \RR^{2}$ be a bounded connected set and let $x \in W$.
If $x \leftrightarrow \partial W$ then
\begin{equation}
\label{eq:chem_distance_small_scales}
D(x, \partial W) \le \frac{2}{\pi} \Vol(W) + 1,
\end{equation}
and consequently
\begin{equation}
\label{eq:distance_small_scales}
\cD(x, \partial W)
    \le \Bigl(\frac{2}{\pi} \Vol(W) + 1\Bigr) \cdot \diam(W).
\end{equation}
\end{lema}

\begin{proof}
Denote by $\{e_i; 1 \le i \le m\}$ the set of all ellipses that intersect
$W$, which is almost surely finite since $W$ is bounded. Since
$x \leftrightarrow \partial W$ there is some point $y \in \partial W$
that can be reached from $x$ by a path contained in $\cE$.
Take a path $\cP$ that connects $x$ and $y$ without self-intersections.
For any fixed ellipse $e$ used by $\cP$, if $\cP \cap e$ is not a straight line
we can reduce the length of $\cP$ by connecting its first and last visit to
that ellipse directly. This modified path may intersect $\partial W$ before 
reaching $y$, but in this case we simply replace $y$ by the first point
of $\partial W$ that was reached. Thus, we can restrict ourselves to
polygonal paths.

Let $f : [0,1] \to \RR^{2}$ be a continuous and injective parametrization
of $\cP$ with $f(0) = x$ and $f(1) = y$, and define $I_j = f^{-1}(e_j)$. By the
properties of $\cP$, we know that each $I_j$ is a closed interval and that
$[0,1] = \cup_{j=1}^{m} I_j$. We build a minimal set of ellipses that covers
$\cP$ by doing a greedy exploration. We can assume that $x \in e_1$ and define
$i_1 := 1$. Then, inductively define $i_{j+1}$ as the
index of an ellipse that intersects $e_{i_{j}}$ and with rightmost point of
$I_{i_{j+1}}$ closest to 1. Since we have a finite collection,
the process ends on some index $i_n$; relabeling if necessary, we can consider
$i_j = j$ for $1 \le j \le n$.

By construction we have that $\cP \subset \cup_{j=1}^{n} e_j$ and
each $e_j$ only intersects $e_{j-1}$ and $e_{j+1}$. 
Finally,
by the same reasoning as in the beginning of the proof we can assume that $e_n$
is the first ellipse to intersect $\partial W$.
We can bound the size of $n$ by using the fact that
$\{e_{i};\; \text{$i$ odd}, 1 \le i \le n-1\}$ and
$\{e_{i};\; \text{$i$ even}, 1 \le i \le n-1\}$ are disjoint collections inside
$W$. Since each $e_i$ contains a ball of radius 1, we have that
\begin{equation*}
    n-1 \le 2 \cdot \frac{\Vol(W)}{\pi},
\end{equation*}
and we proved~\eqref{eq:chem_distance_small_scales}.
The bound on~\eqref{eq:distance_small_scales} follows by using that on each $e_i$
the length of $\cP$ is bounded by $\diam(W)$.
\end{proof}

\subsection{Proof of Theorem~\ref{teo:distance_ellipses_l2}}
\label{sub:proof_of_teo_distance_ellipses}

We now have all the ingredients to bound the Euclidean
distance between distant points inside a same cluster of ellipses.

\begin{proof}[Proof of Theorem~\ref{teo:distance_ellipses_l2}]
We can assume $\delta \in (0, 1 - \frac{2+\alpha}{4})$
and define ${\gamma := \frac{2+\alpha}{4} + \delta}$.
We analyze the probability of $\cD(x,y)$ being large by decomposing
this event with respect to ${\cB_n(a(x), a(y)) \cap \cW_n}$,
obtaining
\begin{align*}
\PP\bigl(x \leftrightarrow y,
    &\,\cD(x,y) > N + N^{\frac{2+\alpha}{4} + \delta}\bigr) \\
    &= \PP\bigl(
        \{x \leftrightarrow y\} \cap \cB_n \cap \cW_n, \,
        \cD(x,y) > N + N^{\gamma}
    \bigr) + o(1),
\end{align*}
since both $\PP \bigl( \cB_n^{\comp} \bigr)$ and
$\PP \bigl( \cB_n \cap \cW_n^{\comp} \bigr)$ tend to zero with $N$ by
Lemmas~\ref{lema:hierarchy} and \ref{lema:control_W_sigma}, respectively.
On event $\{x \leftrightarrow y\} \cap \cB_n \cap \cW_n$
there is a path $\cP$ between $x$ and $y$ and we use $\cP$ to build a path
$\cP'$ as in Definition~\ref{defi:path_cP_prime}.

Using Lemma~\ref{lema:distance_small_scales} we replace the parts of $\cP'$
that use $\cP$ on steps 1. and 4. by a path satisfying the bound
on~\eqref{eq:distance_small_scales}. Actually, we do not lose much by
applying~\eqref{eq:distance_small_scales} at every $W_{\sigma}$, since
\begin{equation*}
\Vol(W_{\sigma})
    \le {\bigl(\tfrac{9}{5}K\bigr)}^2 \cdot \# W_{\sigma}
\quad \text{and} \quad
    \diam(W_{\sigma})
    \le \bigl(\tfrac{9\sqrt{2}}{5} K\bigr) \cdot \# W_{\sigma}.
\end{equation*}
By~\eqref{eq:n_defi_consequences} we have that
$\# W_{\sigma}
    \le \tfrac{1}{3} \uc{c:ellipse_length} \tilde{N}_{n-2}
    \le \tfrac{1}{3} \uc{c:ellipse_length} \cdot 
        e^{(1/\gamma)^{3}(\log^{(2)} \tilde{N})^{\varepsilon}}
$, implying
\begin{equation*}
    \Vol(W_{\sigma}) \cdot \diam(W_{\sigma})
    \le c \cdot K^3 \cdot e^{2/\gamma^{3} \cdot (\log^{(2)} \tilde{N})^{\varepsilon}}.
\end{equation*}

The length of $\cP'$ can be estimated by
\begin{equation*}
l(\cP')
    \le c \sum_{\sigma} \Vol(W_{\sigma}) \cdot \diam(W_{\sigma})
        + \sum_{\sigma'} \diam (E_{\sigma'})
\end{equation*}
where the index in the second sum runs over all highways. 
Since there are $2^{n-1}$ gaps,
the bounds on~\eqref{eq:n_defi_consequences} imply
\begin{equation*}
\sum_{\sigma} \Vol(W_{\sigma}) \cdot \diam(W_{\sigma})
    \le c K^3 \cdot (\log \tilde{N})^{\Delta'}
        e^{2/\gamma^{3} \cdot (\log^{(2)} \tilde{N})^{\varepsilon}}.
\end{equation*}
For the second sum, notice that in the long-range model, for each $0 \le k \le n-2$ there are $2^{k}$ highways of size about $\tilde{N}_k$.
We can therefore write
\begin{align*}
\smash{\sum_{\sigma'}} \diam (E_{\sigma'})
    &\mathmakebox[1cm]{=} \diam (E_{\varnothing}) +
        \smash{\sum_{k=1}^{n-2}} 2^{k} \cdot K \Theta(\tilde{N}_k) \\[3mm]
    &\mathmakebox[1cm]{
        \smash{\stackrel{\eqref{eq:highway_diam}}{=}}
    }
        K \cdot \bigl(|z_{10} - z_{01}| + O(1)\bigr) +
        K 2^{n-1} O(\tilde{N}_1) \\
    &\mathmakebox[1cm]{
        \smash{\stackrel{\eqref{eq:n_defi_consequences}}{=}}}
        K \cdot \bigl( \tilde{N} + \Theta(\tilde{N}_{1})\bigr) +
        K O \bigl((\log \tilde{N})^{\Delta'} \tilde{N}_1\bigr) \\
    &\mathmakebox[1cm]{
        \smash{\stackrel{\eqref{eq:relation_points_sites}}{=}}}
        N + O \bigl((\log N)^{\Delta'} N^{\gamma} \bigr),
\end{align*}
implying the bound
$
l(\cP')
    \le N + O \bigl((\log N)^{\Delta'} N^{\gamma} \bigr).
$
Since $\gamma = \frac{2+\alpha}{4} + \delta$ and $\delta$ can be taken
arbitrarily small, we conclude the proof of~\eqref{eq:distance_ellipses}.
\end{proof}

\section{Bounding chemical distance}
\label{sec:bounding_chemical_distance}

Now we turn to investigating the chemical distance.
The same construction as in the proof of Theorem~\ref{teo:distance_ellipses_l2}
also provides an upper bound for the chemical distance
between $x$ and $y$, since it implies
\begin{align}
D(x,y)
    &\le c \sum_{\sigma} \Vol(W_{\sigma}) + \sum_{\sigma'} 1
    \le c K^2 e^{(1/\gamma)^{3}(\log^{(2)} \tilde{N})^{\varepsilon}} + 2^{n-1}
    \nonumber \\
\label{eq:chem_dist_up_worse}
    &\le (\log |x-y|)^{\Delta' + o(1)}.
\end{align}

However, we can actually achieve a better bound. 
In fact, although our collection of highways provides a structure of long ellipses that links far away points efficiently in terms of their Euclidean distance, it may be conceivable that the optimal strategy to minimize the chemical distance might differ.
An improvement to the bound in~\eqref{eq:chem_dist_up_worse} is given in the next result:
\begin{prop}
\label{prop:chem_dist_up}
For any $\delta > 0$ it holds that
\begin{equation}
\label{eq:chem_dist_up}
\lim_{|x-y| \to \infty}
    \PP
    \Bigl(
        D(x,y) \le \frac{2 + \delta}{\log (2/\alpha)} \cdot \log \log |x-y|
        \Bigm| x \leftrightarrow y 
    \Bigr)
    = 1.
\end{equation}
\end{prop}

\begin{proof}
The main construction for this bound is a way of moving faster than through
highways, see Figure~\ref{fig:chem_up_bound}. This construction has no
counterpart on discrete long-range model, since it leverages on the
property that two ellipses that cross in their middle section are
connected.

We consider a sequence $(l_n; n \geq 0)$ of increasing lengths which is defined
recursively by $l_n = l_{n-1}^{2/\alpha}(\log l_{n-1})^{-1}$. The value of
$l_0$ is fixed later. Consider also the following collection of boxes
\begin{equation*}
B_n =
    \begin{cases}
    [0, l_n]     \times [0, l_{n-1}]	& \text{if $n$ odd} \\
    [0, l_{n-1}] \times [0, l_{n}]      & \text{if $n$ even}
    \end{cases}
\end{equation*}
and define event $A_n$ in which box $B_n$ is crossed in its
longest direction by one ellipse. By Lemma~\ref{lema:cross_box_one_ellipse}
and our choice of sequence $(l_n)$ we have
\begin{equation*}
\sum_{n\geq 1} \PP(A_n^{\comp})
    \le \sum_{n\geq 1}
        \exp\bigl[ - u \uc{c:cross_one_ellipse}^{-1}
            l_n^{-\alpha} l_{n-1}^{2}\bigr]
    = \sum_{n\geq 1}
        \exp\bigl[ - u \uc{c:cross_one_ellipse}^{-1}
            (\log l_{n-1})^{\alpha}\bigr]
\end{equation*}
Now, we check that the series above converges by estimating the growth rate of 
sequence $(l_n)$. Notice that
$l_n \le l_{n-1}^{2/\alpha} \le l_0^{(2/\alpha)^{n}}$ and also that
this upper bound implies
\begin{align*}
l_n
    = l_{n-1}^{2/\alpha} \cdot (\log l_{n-1})^{-1}
    &\geq l_{n-1}^{2/\alpha} \cdot (2/\alpha)^{-(n-1)} (\log l_0)^{-1} \\
    &\geq l_{n-2}^{(2/\alpha)^{2}} (2/\alpha)^{-(n-1) - (2/\alpha) (n-2)} (\log
    l_0)^{-(1 + (2/\alpha))} \\
    &\geq
    l_{0}^{(2/\alpha)^{n}}
    (2/\alpha)^{-\sum_{j=1}^{n} (n-j)(2/\alpha)^{j-1}}
    (\log l_0)^{-\sum_{j=1}^{n} (2/\alpha)^{j-1}}.
\end{align*}
Computing the sums on the last line one obtains as $n \to \infty$ that
\begin{align*}
\sum_{j=1}^{n} (n-j)(2/\alpha)^{j-1}
    &= (2/\alpha)^{n}
        \bigl( \tfrac{2/\alpha}{((2/\alpha) - 1)^2} + o(1) \bigr), \\
\smash{\sum_{j=1}^{n}} (2/\alpha)^{j-1}
    &= (2/\alpha)^{n} \bigl( \tfrac{1}{(2/\alpha) - 1} + o(1) \bigr),
\end{align*}
which leads to
\begin{equation*}
l_n
    \geq
    \exp\Bigl[
        \Bigl(\frac{2}{\alpha}\Bigr)^{n}
        \Bigl\{ \log l_0
            - \frac{(2/\alpha)\log (2/\alpha)}{((2/\alpha) - 1)^2}
            - \frac{\log \log l_0}{(2/\alpha) - 1}
            + o(1)
        \Bigr\}
        \Bigr]
\end{equation*}
and thus for some large $l_0 = l_0(\alpha)$ the coefficient in curly
brackets can be estimated from below by $2 + o(1)$, which implies
\begin{equation*}
\sum_{n\geq 1} \PP(A_n^{\comp})
    \le \sum_{n\geq 1}
        \exp\bigl[ - u c \bigl((2/\alpha)^{\alpha}\bigr)^{n-1}\bigr]
    < \infty
\end{equation*}
since $(2/\alpha)^{\alpha} > 1$. This means we can make $\PP(\cap_{n \geq n_0} A_n)$
arbitrarily close to one by taking $n_0$ sufficiently large. Notice that on
this event we move faster than when using highways, since we can get from
$B_{n_0}$ to distance $l_n$ using only $n - n_0$ ellipses.

Besides faster highways, we also build a useful collection of circuits.
Let $U^{0}_n$ be the event in which box
$[-2l_n, 2l_n] \times [l_n, 2l_n]$ is crossed in its longest direction
with one ellipse and let $U^{j}_n$ be the analogous events obtained by
rotating this box counterclockwise by $j \cdot \pi/2$, $j=1, 2, 3$.
Defining events $C_n := \cap_{j=1}^4 U^{j}_{n}$, we have by
Lemma~\ref{lema:cross_box_one_ellipse}
\begin{equation*}
\sum_{n\geq 1} \PP(C_n^{\comp})
    \le \sum_{n\geq 1} 4 \PP\bigl((U^{0}_{n})^{\comp}\bigr)
    \le 4 \sum_{n\geq 1} \exp[ -u c \cdot l_n^{2-\alpha} ] < \infty.
\end{equation*}
Moreover, on $C_n$ we have a circuit made of four ellipses, whose
supporting lines form a convex quadrilateral $Q_n$ that surrounds 
$[-l_{n}, l_{n}]^2$ but stays inside $[-2l_{n_0}, 2l_{n_0}]^2$.

Now we are ready to prove~\eqref{eq:chem_dist_up}.
Without loss of generality, we can assume $y$ is the origin.
Fix any $\varepsilon > 0$ and choose $n_0$ sufficiently large such that
\begin{equation*}
    \PP \bigl( \cap_{n \geq n_0} A_n \bigr) \geq 1 - \varepsilon
    \quad \text{and} \quad
    \PP \bigl( \cap_{n \geq n_0} C_n \bigr) \geq 1 - \varepsilon.
\end{equation*}
Let us also define $A_n(x)$ and $C_n(x)$ as the events analogous to
$A_n$ and $C_n$ but considering that $x$ is the origin.
Thus, if we define event
\begin{equation*}
V := \smash{
        \Bigl(\bigcap_{n \geq n_0} A_n\Bigr)    \cap 
        \Bigl(\bigcap_{n \geq n_0} C_n\Bigr)    \cap 
        \Bigl(\bigcap_{n \geq n_0} A_n(x)\Bigr) \cap 
        \Bigl(\bigcap_{n \geq n_0} C_n(x)\Bigr)},
\end{equation*}
we can write
\begin{equation*}
\bigl|\PP(o \leftrightarrow x) - \PP(\{o \leftrightarrow x\} \cap V)\bigr|
    \le \PP(V^{\comp}) \le 4 \varepsilon.
\end{equation*}
On event $\{o \leftrightarrow x\} \cap V$ we have
some path $\cP$ of ellipses connecting $o$ to $x$. For $|x| > 2 l_{n_0}$ path
$\cP$ intersects quadrilaterals $Q_{n_0}$ and $Q_{n_0}(x)$.

Let us define $n_1 = n_1(x)$ as the first index $k$
such that $x + B_{n_0} \subset [-l_{k}, l_{k}]^2$. We have that
\begin{equation*}
l_{n_1} \geq |x| + l_{n_0} > l_{n_1 - 1},
\quad \text{implying}\quad
    n_1 = \frac{\log \log |x|}{\log (2/\alpha)} + O(1).
\end{equation*}
Finally, notice that event $\cap_{n \geq n_0} A_n$ ensures $Q_{n_0}$
is connected to $Q_{n_1}$ by a path $P$ of at most $n_1$ ellipses.
We can also find a path $P(x)$ of at most $n_1$ ellipses connecting
$Q_{n_0}(x)$ to $Q_{n_1}$, when we consider event $\cap_{n \geq n_0} A_n(x)$.
Thus, we can bound the chemical distance of $o$ and $x$ by the number of
ellipses in the following path $\cP'$:
\begin{enumerate}[(i)]
\item Move from $o$ to $Q_{n_0}$ using the minimal number of ellipses and then
    follow circuit $Q_{n_0}$ till meeting $P \cap Q_{n_0}$.
\item Move from $P \cap Q_{n_0}$ to $P \cap Q_{n_1}$ and then follow circuit
    $Q_{n_1}$ till you meet $P(x) \cap Q_{n_1}$. Move from
    $P(x) \cap Q_{n_1}$ to $P(x) \cap Q_{n_0}(x)$.
\item Follow circuit $Q_{n_0}(x)$ till you meet a point of $Q_{n_0}(x)$
    connected to $x$ by a path inside $Q_{n_0}(x)$ that uses a
    minimal number of ellipses.
\end{enumerate}
By Lemma~\ref{lema:distance_small_scales} we can bound the number of ellipses
used in each of the steps (i) and (iii) by
$\frac{2}{\pi} (2l_{n_0})^{2} + 1$. Moreover, moving between points in a same
quadrilateral takes at most $4$ ellipses. This leads to
\begin{equation*}
D(o, x)
    \le 2 \cdot \Bigl(\frac{2}{\pi} \cdot 4l_{n_0}^{2} + 1\Bigr) + 3 \cdot 4 + 2 n_1
    \le \frac{2}{\log (2/\alpha)} \log \log |x| + O(1).
\end{equation*}
Taking the limit as $|x|$ tends to infinity, one obtains for arbitrary
$\varepsilon > 0$ that
\begin{equation*}
\varlimsup_{|x| \to \infty}
\PP
    \Bigl(
    x \leftrightarrow o,\ 
    D(x,o) > \frac{2 + \delta}{\log (2/\alpha)} \cdot \log \log |x|
    \Bigr)
    \le 4 \varepsilon. \qedhere
\end{equation*}
\end{proof}

\begin{figure}
\centering
\begin{tikzpicture}[scale=0.31]
\clip (-3.5,-3) rectangle (18,19);
\draw (-16,-16) rectangle ++(32,32);

\draw[very thick, green] (17, -3) --
    (16.5,17) node[black, above] {$Q_{n_1(x)}$} -- (-3, 17.5);

\coordinate (o) at (0,0);
\coordinate (x) at (7,9);
\foreach \site in {o, x}{
\begin{scope}[shift={(\site)}]
    \foreach \x in {0, 1, 2, 3}{
        \draw[dashed, rotate={90*\x}] (-2,1) rectangle ++(4,1);}
    \draw[thick] (-1,-1) rectangle ++(2,2);
    \filldraw (\site) circle (.15) node[below] {$\site$};
    \draw[very thick, green] (-1.7, -1.5) -- (1.3, -1.8) --
        (1.2, 1.6) -- (-1.8, 1.5) -- cycle;
    \draw[very thick, blue]
        (0, .5) -- (3, .7) (1.8,0) -- (1.3,6) (0,3) -- (18,2.6);
\end{scope}}
\draw [very thick, decorate, decoration={name=zigzag, amplitude=1.5pt}] {
    (o) -- (1.6, .6) -- (1.5,3) -- (17,2.6) -- (16.5, 11.8) -- (8.5, 12)
    -- (8.6, 9.6) -- (x)};
\node at (-1.8, 2.8) {$Q_{n_0}$};
\node at ($(-1.8, 2.8) + (x)$) {$Q_{n_0}(x)$};
\node at (13, 1.7) {$P$};
\node at (13, 10.7) {$P(x)$};
\end{tikzpicture}%
\hspace{2mm}%
\begin{tikzpicture}[scale=0.5]
\clip (-2,-1.4) rectangle (9.5,11);
\draw (0,0) rectangle (2,1);
\draw (0,0) rectangle (2,4);
\draw (0,0) rectangle (8,4);
\draw (0,0) rectangle (8,16);
\fill[blue, opacity=.7, rotate around={10:(1,.6)}]
    (1,.6) circle (1.4 and .1);
\fill[blue, opacity=.7, rotate around={92:(0.5,2)}]
    (0.5,2) circle (2.8 and .1);
\fill[blue, opacity=.7, rotate around={-15:(4,2)}]
    (4,2) circle (5.6 and .1);
\fill[blue, opacity=.7, rotate around={80:(6,4)}]
    (6,4) circle (11 and .1);
\node[left] at (0,1) {$l_{n_0 - 1}$};
\node[left] at (0,4) {$l_{n_0 + 1}$};
\node[left] at (0,8) {$\vdots$};
\node[below] at (2,0) {$l_{n_0}$};
\node[below] at (8,0) {$l_{n_0 + 2}$};
\node[below] at (9,0) {$\ldots$};
\end{tikzpicture}
\caption{Construction of short path $\cP'$, depicted by a zigzag line.
On the right, we show event $\cap_{n \geq n_0} A_n$ in which we have an
`improved highway'. On the left, the improved highways $P$ and $P(x)$ are
connected to quadrilaterals $Q_{n_0}, Q_{n_0}(x)$ and $Q_{n_1(x)}$ to form
$\cP'$.}
\label{fig:chem_up_bound}
\end{figure}
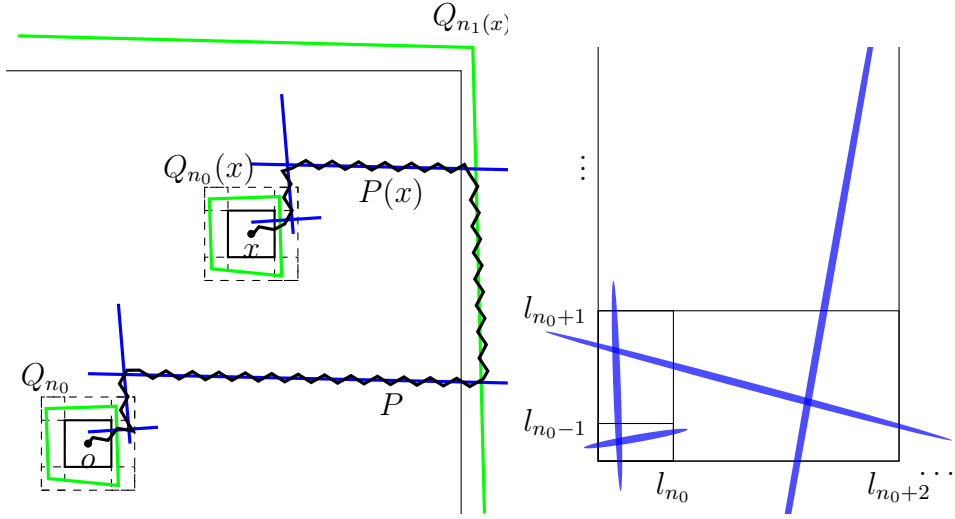

\subsection{Lower bound for chemical distance}
\label{sub:lower_bound_for_chemical_distance}

The argument from~\cite{biskup2011graph}, due originally to
Trapman~\cite{trapman2010growth}, cannot provide us a lower bound
for the chemical distance since we already have an upper bound for
$D(x,y)$ of order $\log \log |x-y|$. It is possible to employ a
similar strategy based on BK
inequality~\cite{van1985inequalities, van1996note}, but here
we are able to use a more elementary approach.

For $0 < l_1 < l_2$, we make a slight abuse of notation by denoting the
chemical distance between sets $B(l_1)$ and $\partial B(l_2)$ by
$D(l_1, l_2)$. Instead of working with $D(o,x)$ directly,
we investigate the quantity $D(1, |x|)$.
\begin{prop}
\label{prop:chem_dist_low_induction}
Let $\gamma := \frac{\alpha - 1}{2}$ and define
$C := \max\{15, u 2^{\alpha -1} \uc{c:annulus_bound}\}$, where $\uc{c:annulus_bound}$
is given by Lemma~\ref{lema:annulus_bound}.
For every $n \geq 0$ we have
\begin{equation}
\label{eq:chem_dist_bound}
\PP(D(1, |x|)
    \le 2^{n}) \le C^{n} |x|^{- 2 \gamma^{2^{n}}},
    \quad \text{for $|x| > 2^{\gamma^{1 -n - 2^{n}}}$}.
\end{equation}
\end{prop}

\begin{proof}
The proof is by induction.
By Lemma~\ref{lema:annulus_bound} we have for $|x| > 2$ that
\begin{equation*}
\PP(D(1, |x|) = 1)
    \le u \uc{c:annulus_bound} (|x| - 1)^{1-\alpha}
    \le u \uc{c:annulus_bound} 2^{\alpha - 1} |x|^{1-\alpha}
    \le C |x|^{- 2 \gamma},
\end{equation*}
and we proved case $n=0$.
For the induction step, we notice that for $n \geq 0$
\begin{align*}
\PP(D(1, |x|) \le 2^{n+1})
    &\le \PP(D(1, l) \le 2^{n}) + \PP(D(l, |x|) \le 2^{n})\\
    &\le \PP(D(1, l) \le 2^{n}) +
        (7l) \cdot \PP(D(1, |x|-l) \le 2^{n}),
\end{align*}
where in the last inequality we used union bound and the fact that any
$\partial B(l)$ with $l > 2$ can be covered by at most $7l$ balls of
radius $1$ and some of these balls must be connected to $\partial B(|x|)$.
If the induction hypothesis can be applied, the choice
\begin{equation*}
l = \exp\Bigl[
    (\log |x|) \cdot
    \frac{2 \gamma^{2^{n}}}{1 + 2 \gamma^{2^{n}}}
    \Bigr]
\end{equation*}
ensures that the two terms in the sum above are approximately the
same size, since it makes
\begin{equation*}
l^{-2 \gamma^{2^{n}}}
    = l \cdot |x|^{-2 \gamma^{2^{n}}}
    = \exp\Bigl[ 
        (\log |x|) \cdot
        \frac{-4 \gamma^{2^{n+1}}}{1 + 2 \gamma^{2^{n}}}
        \Bigr].
\end{equation*}
Let us denote $b_n = 2^{\gamma^{1 - n - 2^{n}}}$ and suppose
$|x| > b_{n+1}$. Then, it is easy to check that our choice of
$l$ satisfies $l > b_{n}$ if and only if
\begin{equation}
\label{eq:low_bound_induction}
\exp
    \Bigl[
    (\log 2) \gamma^{- n - 2^{n+1}} \cdot
    \frac{2 \gamma^{2^{n}}}{1 + 2 \gamma^{2^{n}}}
    \Bigr]
\geq
    \exp
    \Bigl[
    (\log 2) \gamma^{1 - n - 2^{n}}
    \Bigr].
\end{equation}
Inequality~\eqref{eq:low_bound_induction} is equivalent to
\begin{equation*}
\gamma^{-1-2^{n}}
    = \gamma^{- n - 2^{n+1} - (1 - n - 2^{n})}
    \geq 1 + \frac{1}{2\gamma^{2^{n}}},
\quad \text{or} \quad
1
    \geq \gamma^{1 + 2^{n}} + \frac{\gamma}{2}.
\end{equation*}
Since $\gamma = \frac{\alpha -1}{2} \in (0, 1/2)$ and $n \geq 0$,
we have that~\eqref{eq:low_bound_induction} is satisfied. Analogously,
we must check that $|x| - l \geq b_{n}$. This can be done by noticing
\begin{equation*}
\frac{l}{|x|}
    = |x|^{- \frac{1}{1 + 2 \gamma^{2^{n}}}}
    < \exp\Bigl[
        - (\log 2)
        \frac{\gamma^{- n - 2^{n+1}}}{1 + 2\gamma^{2^{n}}}
        \Bigr]
    \le 2^{ - \frac{1}{\gamma^{2} + 2\gamma^{3}} }
    \le 2^{-2},
\end{equation*}
which implies that 
$|x| - l \geq \frac{3}{4}|x| \geq \frac{3}{4} b_n > b_{n-1}$. Thus, we are
allowed to use the induction hypothesis. Using the bound
\begin{equation*}
{\bigl(1 - l/|x|\bigr)}^{-\smash{2\gamma^{2^{n}}}}
    \le \bigl(1 - 1/4 \bigr)^{-\smash{2\gamma^{2^{n}}}}
    \le (4/3)^{2\gamma}
    < 2,
\end{equation*}
we can write
\begin{align*}
\PP(D(1, |x|) \le 2^{n+1})
    &\le C^{n} l^{-2 \gamma^{2^{n}}} +
        7l \cdot C^{n} (|x| - l)^{-2 \gamma^{2^{n}}}
    \le \bigl(1 + 14 \bigr) \cdot C^{n} \cdot l^{-2 \gamma^{2^{n}}} \\
    &\le C^{n+1}
        \exp\Bigl[ 
        (\log |x|) \cdot
        \frac{-4 \gamma^{2^{n+1}}}{1 + 2 \gamma^{2^{n}}}
        \Bigr].
\end{align*}
Finally, we can estimate
\begin{equation*}
\frac{
\PP(D(1, |x|) \le 2^{n+1})
}{
C^{n+1} |x|^{- 2 \gamma^{2^{n+1}}}
}
    \le \exp\Bigl[ 
        (\log |x|) \cdot 2 \gamma^{2^{n+1}}
        \Bigl(
        1 - \frac{2}{1 + 2 \gamma^{2^{n}}}
        \Bigr)
        \Bigr]
    \le 1,
\end{equation*}
since $1 - \frac{2}{1 + 2 \gamma^{2^{n}}} \in (-1, 0)$.
\end{proof}

From Proposition~\ref{prop:chem_dist_low_induction}, we can conclude the
following
\begin{coro}
Let $0 < \delta < \frac{1}{\log (1/\gamma)}$, where
$\gamma := \frac{\alpha - 1}{2}$. Then we have
\begin{equation}
\lim_{|x| \to \infty}
    \PP\bigl(D(1, |x|) \le \delta \log \log |x|\bigr) = 0.
\end{equation}
\end{coro}

\begin{proof}
Let us choose
\vspace{-3mm}
\begin{equation*}
n = \lfloor
    (\log 2)^{-1} \cdot \bigl(\log \delta + \log \log \log |x|\bigr)
    \rfloor,
\end{equation*}
which makes
$2^{n} \in \bigl[\frac{\delta}{2} \log \log |x|, \delta \log \log |x|\bigr]$.
Notice that $|x| > 2^{\gamma^{1 - n - 2^{n}}}$ for large $|x|$, since
$1/\delta > \log (1/\gamma)$ implies
\begin{align*}
\log \log 2^{\gamma^{1 - n - 2^{n}}}
    &= \log \log 2 + (2^{n}+n-1) \log (1/\gamma) \\
    &\le \frac{2^{n}}{\delta} \le \log \log |x|.
\end{align*}
Hence, by Proposition~\ref{prop:chem_dist_low_induction} we have
\begin{align*}
\PP(D(1, |x|) \le 2^{n})
    &\le C^{n} |x|^{-2 \gamma^{2^{n}}} \\
    &\le \exp\bigl[
        \log C \cdot n - 2 \log |x| \cdot \gamma^{\delta \log \log |x|}
    \bigr] \\
    &\le \exp\biggl[
        \frac{\log C}{\log 2} \cdot \log \log \log |x| + O(1)
        - 2 (\log |x|)^{1 - \delta \log (1/\gamma)}
    \biggr]
\end{align*}
    which tends to zero since $\delta \log (1/\gamma) < 1$.
\end{proof}

\bibliographystyle{plain}
\bibliography{distance_ellipses}
\end{document}